\documentclass[11pt]{amsart}

\allowdisplaybreaks        

\usepackage{geometry}  
\newgeometry{left=1.4in, right=1.4in, bottom=1in, top=1in}

\usepackage{amsmath}
\usepackage{amsfonts}
\usepackage{amssymb}
\usepackage{amsthm}
\usepackage{mathtools} 
\usepackage{latexsym}

\usepackage{enumerate}
\usepackage{cancel}
\usepackage{cases}
\usepackage{empheq}
\usepackage{multicol}

\usepackage{wrapfig}

\usepackage{caption}
\usepackage{subcaption}

\usepackage{tikz}
\usepackage{pgf}
\usepackage{pgfplots} 
\pgfplotsset{compat=1.16}

\usepackage{color}

\usepackage[backgroundcolor=gray!30,linecolor=black]{todonotes}
   
\usepackage{mathrsfs}
\usepackage{fontenc} 
\usepackage{inputenc} 

\usepackage{verbatim}

\theoremstyle{plain}
\newtheorem{theorem}{Theorem}[section]

\newtheorem{lemma}[theorem]{Lemma}
\newtheorem{corollary}[theorem]{Corollary}

\theoremstyle{definition}
\newtheorem{definition}[theorem]{Definition}

\theoremstyle{remark}
\newtheorem{remark}[theorem]{Remark}

\numberwithin{equation}{section} 
\numberwithin{figure}{section}   

\usepackage[square,comma,numbers,sort&compress]{natbib}
\usepackage{url}

\newcommand{\vect}[1]{\mathbf{#1}}

\newcommand{\bk}{\vect{k}}
\newcommand{\bphi}{\boldsymbol{\phi}}
\newcommand{\bpsi}{\boldsymbol{\psi}}

\newcommand{\bu}{\vect{u}}

\newcommand{\bv}{\vect{v}}
\newcommand{\bV}{\vect{V}}
\newcommand{\bw}{\vect{w}}

\newcommand{\bx}{\vect{x}}

\newcommand{\bg}{\vect{g}}
\newcommand{\bh}{\vect{h}}

\newcommand{\bbf}{\vect{f}}

\newcommand{\field}[1]{\mathbb{#1}}
\newcommand{\nN}{\field{N}}

\newcommand{\nR}{\field{R}}

\newcommand{\lp}{\left(}
\newcommand{\rp}{\right)}

\newcommand{\tbu}{\widetilde{\bu}}

\newcommand{\oldnu}{\nu}

\newcommand{\pd}[2]{\frac{\partial #1}{\partial #2}}

\newcommand{\set}[1]{\left\{#1\right\}}
\newcommand{\ip}[2]{\left<#1,#2\right>}

\newcommand{\pnt}[1]{\left(#1\right)}

\usepackage{placeins} 

\newcounter{my_counter}
\title[Data Assimilation: Nonlinear algorithm]{Super-exponential convergence rate of a nonlinear continuous data assimilation algorithm: The 2D Navier-Stokes equations paradigm}

\date{\today}

\author{Elizabeth Carlson}
\address[Elizabeth Carlson]{Department of Mathematics, 
                University of Victoria,
                Victoria, BC V8P 5C2, Canada
        }
\email[Elizabeth Carlson]{ecarlson10@uvic.ca}
\author{Adam Larios}
\address[Adam Larios]{Department of Mathematics, 
                University of Nebraska--Lincoln,
        Lincoln, NE 68588-0130, USA}
\email[Adam Larios]{alarios@unl.edu}
\author{Edriss~S.~Titi}
\address[Edriss~S.~Titi]{Department of Mathematics, Texas A\&M University, College Station, TX 77843, USA and Department of Applied Mathematics and Theoretical Physics, University of Cambridge, Cambridge CB3 0WA, UK}
\email[Edriss~S.~Titi]{$\quad$titi@math.tamu.edu \\ and edriss.titi@damtp.cam.ac.uk}


\begin{document}

\begin{abstract}
We study a nonlinear-nudging modification of the Azouani-Olson-Titi continuous data assimilation (downscaling) algorithm for the 2D incompressible Navier-Stokes equations.  We give a rigorous proof that the nonlinear-nudging system is globally well-posed, and moreover that its solutions converge to the true solution exponentially fast in time.  Furthermore, we also prove that, once the error has decreased below a certain order one threshold, the convergence becomes double-exponentially fast in time, up until a precision determined by the sparsity of the observed data.  In addition, we demonstrate the applicability of the analytical and sharpness of the results computationally.
\end{abstract}

\maketitle
\thispagestyle{empty}

\noindent

\section{Introduction}
\noindent
Many dissipative dynamical systems that model physical processes are chaotic and highly sensitive to initial conditions.  Hence, having incomplete information about initial conditions makes simulating these systems accurately a difficult task.
To overcome this issue in practice, the available spatially discrete observed data can be used to inform the model via a wide variety of techniques, collectively known as \emph{data assimilation}.  Data assimilation can be done using a variety of different methods incorporate observations into the mathematical model generally using either statistical or continuous techniques.  In this paper, we focus on a continuous data assimilation (CDA) algorithm, also known as the Azouani-Olson-Titi (AOT) algorithm.  The CDA algorithm is based on the mathematical theory that many dissipative evolution equations describing fluid flow have solutions that are, in large-time, determined uniquely by the values of their solutions at a finite number of adequately distributed nodes or modes (see, e.g. \cite{Foias_Manley_Rosa_Temam_2001} and references therein).  It incorporates observational data into the model at the partial differential equation (PDE) level using a feedback control (nudging) term.  This paper investigates the convergence of a nonlinear-nudging version of the CDA algorithm.

The CDA algorithm was first introduced in \cite{Azouani_Olson_Titi_2014,Azouani_Titi_2014} (see also \cite{Cao_Kevrekidis_Titi_2001,Hayden_Olson_Titi_2011,Olson_Titi_2003,Olson_Titi_2008_TCFD} for early ideas in this direction).  The algorithm considers a dissipative dynamical system
\begin{align}\label{generic_DE}
 \frac{d\bu}{dt} = \mathcal{F}(\bu),
\end{align}
with an unknown initial condition.  We denote a given interpolation of the observations of the unknown reference solution $\bu$ at course spatial scales by $I_h(\bu)$, where $h$ is some characteristic length scale of the the observational data (e.g., the average spatial distance between observations).
These observations are incorporated via a feedback control term in the following modified system
\begin{subequations}\label{generic_DE_CDA}
\begin{alignat}{2}
 \frac{d\bv}{dt} &= \mathcal{F}(\bv) + \mu (I_h(\bu)-I_h(\bv)),\\
   &\hspace{0.4in} \bv(t=0) = \bv_0,
 \end{alignat}
\end{subequations}
where $\mu > 0$ is an adequately chosen positive relaxation (nudging) parameter and $\bv_0$ is any sufficiently smooth initial condition.  A wide class of standard interpolants $I_h$ are admissible by the analysis of \cite{Azouani_Olson_Titi_2014} including, e.g., piecewise constant interpolation, linear interpolation, and Fourier truncation, among others, making this algorithm very adaptable for physical models and computationally inexpensive to implement.  In the context of the 2D incompressible Navier-Stokes equations with both no-slip and periodic boundary conditions,  the global well-posedness of \eqref{generic_DE_CDA} and exponential convergence in time to the reference solution $\bu$ of \eqref{generic_DE} were proven in \cite{Azouani_Olson_Titi_2014}.  This algorithm was then investigated in the context of numerous dissipative dynamical systems under a variety of assumptions including noisy data, incorrect parameters, incorrect models, data provided discretely in time, and assimilation of only some instead of all state variables, in, e.g., 
\cite{Akbas_Cibik_2020,
Albanez_Nussenzveig_Lopes_Titi_2016,
Altaf_Titi_Knio_Zhao_Mc_Cabe_Hoteit_2015,
Balakrishna_Biswas_2022,
Bessaih_Ginting_McCaskill_2022,
Bessaih_Olson_Titi_2015,
Biswas_Bradshaw_Jolly_2020,
Biswas_Foias_Mondaini_Titi_2018downscaling,
Biswas_Hudson_Larios_Pei_2017,
Biswas_Martinez_2017,
Biswas_Price_2020_AOT3D,
Carlson_Hudson_Larios_2020,
Carlson_Hudson_Larios_Martinez_Ng_Whitehead_2021,
Carlson_Larios_2021_sens,
Carlson_VanRoekel_Petersen_Godinez_Larios_2021,
Celik_Olson_Titi_2019,
Chen_Li_Lunasin_2021,
Cao_et_al_2022,
Chow_Leung_Pakzad_2022,
Diegel_Rebholz_2021,
DiLeoni_Clark_Mazzino_Biferale_2018_inferring,
Desamsetti_Dasari_Langodan_Knio_Hoteit_Titi_2019_WRF,
Desamsetti_et_al_2022,
Du_Shiue2021,
Farhat_GlattHoltz_Martinez_McQuarrie_Whitehead_2019,
Farhat_Johnston_Jolly_Titi_2018,
Farhat_Jolly_Titi_2015,
Farhat_Lunasin_Titi_2016abridged,
Farhat_Lunasin_Titi_2016benard,
Farhat_Lunasin_Titi_2016_Charney,
Farhat_Lunasin_Titi_2017_Horizontal,
Farhat_Lunasin_Titi_2018_Leray_AOT,
Franz_Larios_Victor_2022,
Foias_Mondaini_Titi_2016,
Foyash_Dzholli_Kravchenko_Titi_2014,
GarciaArchilla_Novo_2020,
GarciaArchilla_Novo_Titi_2018,
Gardner_Larios_Rebholz_Vargun_Zerfas_2020_VVDA,
Gesho_Olson_Titi_2015,
GlattHoltz_Kukavica_Vicol_2014,
Hudson_Jolly_2019,
Ibdah_Mondaini_Titi_2018uniform,
Jolly_Martinez_Olson_Titi_2018_blurred_SQG,
Jolly_Martinez_Titi_2017,
Larios_Pei_2017_KSE_DA_NL,
Larios_Pei_2018_NSV_DA,
Larios_Rebholz_Zerfas_2018,
Larios_Victor_2019,
Lunasin_Titi_2015,
Markowich_Titi_Trabelsi_2016_Darcy,
Martinez_2022_pr,
Martinez_2022_forcing,
Mondaini_Titi_2018_SIAM_NA,
Pachev_Whitehead_McQuarrie_2021concurrent,
Pei_2019,
Rebholz_Zerfas_2018_alg_nudge,
Rodrigues_2022,
Titi_Trabelsi_2022,
You_2022,
You_Xia_2022,
Zauner_Mons_Marquet_Leclaire_2022,
Zerfas_Rebholz_Schneier_Iliescu_2019}
and the references therein.).  In each of these papers exponential convergence either to $0$ or up to a certain measurable error regardless of the choice of initial conditions, and the slight modification required to existing models makes the CDA algorithm an efficient and effective data assimilation algorithm.  Classical data assimilation methods are generally statistical in nature, including the Kalman filter and its variants as well as 4DVAR, but these methods are non-trivial to implement and computationally much more expensive than simply running a simulation of the dynamical system alone, making the CDA algorithm a more efficient and potentially viable alternative for use in certain real world models (see, e.g., \cite{Desamsetti_Dasari_Langodan_Knio_Hoteit_Titi_2019_WRF, Desamsetti_et_al_2022, Carlson_VanRoekel_Petersen_Godinez_Larios_2021}).

The motivation for this work comes from  the computational study \cite{Larios_Pei_2017_KSE_DA_NL}, which introduced and investigated a nonlinear version of the CDA algorithm in the context of the Kuramoto-Sivashinsky equations.
This nonlinear-nudging algorithm computationally demonstrated \textit{super-exponential} convergence in time to the reference solution for the 1D Kuramoto-Sivashinsky equations.  This was later demonstrated with a similar modification in a computational study on the 2D magnetohydrodynamic equations in \cite{Hudson_Jolly_2019}.  In \cite{Du_Shiue2021}, the authors adapted the nonlinear-nudging data assimilation schemes of \cite{Larios_Pei_2017_KSE_DA_NL} to the context of the Lorenz equation, and proved exponential (but not super-exponential) convergence. Note that another nonlinear approach to nudging was proposed and studied in \cite{Germano_2017}, but using a very different method from that in the present work.
In our case, in order to simplify the practical implementation, we consider the following nonlinear-nudging system of equations:
\begin{subequations}
\begin{align}\label{generic_DE_nonlin_CDA_sum_CDA}
 \bv_t &= \mathcal{F}(\bv) + \mu\mathcal{N}(I_h(\bu)-I_h(\bv)) + \beta(I_h(\bu)-I_h(\bv)), \\
 \bv(\bx,0) &= \bv_0(\bx),
 \end{align}
\end{subequations}
where we denote, with $\gamma \in [0,1)$,
\begin{align*}
\mathcal{N}(\bphi) := 
\begin{cases}
0, &  \text{ if }\quad\|\bphi\|_{L^2(\Omega)} = 0,
\\ 
\bphi\|\bphi\|_{L^2(\Omega)}^{-\gamma}, & \text{ if }\quad \|\bphi\|_{L^2(\Omega)}>0.
\end{cases} 
\end{align*}
%

\noindent Note that we formally recover the linear-nudging CDA algorithm when $\gamma=0$.  As demonstrated computationally in \cite{Hudson_Jolly_2019,Larios_Pei_2017_KSE_DA_NL}, we expect that once $\|I_h(\bu)-I_h(\bv)\|_{L^2(\Omega)}<1$, the error of the nonlinear-nudging algorithm should enjoy a super-exponential decay rate and reach machine precision at an earlier time than the linear-nudging algorithm.

We anticipate that this formulation should yield double exponential convergence of the algorithm if $\gamma \in (0,1)$.
The paradigm equations used to demonstrate the convergence of this combined linear/nonlinear-nudging algorithm are the 2D incompressible Navier-Stokes equations.  Indeed we prove here, with certain reasonable assumptions on $I_h$ and given a sufficiently developed reference flow, a double-exponential decay rate of the error, at least down to a level $\epsilon>0$, determined by $h$ and the $L^2$ norm (the size) of the initial data of the reference solution $\bu$ (but otherwise independent of the initial data), and other physical parameters in the system (see Theorem \ref{first_conv_theorem} below for details).
We also prove that after the precision $\epsilon$ is reached, which happens in finite time, the error continues to decay to zero at least at an exponential rate.  

We believe the $\epsilon$ barrier for super-exponential convergence discussed above is likely insurmountable due to 1) a direct observation from the method of proof (see Remark \ref{remark_eps_barrier}), and 2) a heuristic argument of the same observation for a more general dissipative system in Appendix \ref{sec_Heuristic_eps_argument}.  Namely, as can be seen by the arguments Appendix \ref{sec_Heuristic_eps_argument},  the nonlinear term $\mathcal{N}(I_h(\bu)-I_h(\bv))$ in \eqref{generic_DE_nonlin_CDA_sum_CDA} forces the large spatial scales (e.g., low modes) to converge at a super-exponential rate; however, this process seems to eventually destabilize the smaller spatial scales so much that they cannot be suppressed and stabilized by the linear viscous effect, obstructing super-exponential convergence after error becomes sufficiently small.  Therefore, we expect to see super-exponential convergence rates for early times, which then become merely exponential for later times once the error becomes very small.  Indeed, this is what we observe in the simulations in Section \ref{sec_simulations}.

The paper is organized as follows: in Section \ref{secPre} we lay out notation and state definitions and preliminary theorems and results for reference; in Section \ref{secExis} we prove a global well-posedness result for the nonlinear-nudging system \eqref{nonlinDA};  in Section \ref{secConv} we prove the convergence results discussed above; in Section \ref{sec_simulations} we investigate our results in simulations.  Concluding remarks are in Section \ref{sec_conclusion}.


\section{Preliminaries}\label{secPre}

\noindent
The convergence of the nonlinear-nudging CDA algorithm will be proved in this paper in the context of the 2D incompressible Navier-Stokes equations, as a paradigm, with periodic boundary conditions.  However, the result is equally valid for general dissipative systems of equations with physical boundary conditions.  We begin by stating some preliminary theoretical foundations.  First, the initial-boundary value problem
\begin{subequations}\label{NSEtrue}
\begin{align}
 \bu_t + \bu\cdot\nabla\bu - \nu\triangle\bu + \nabla p &= \bbf, &&\text{in }  \Omega\times (0,T), \\
 \nabla\cdot\bu &= 0, &&\text{in } \Omega\times (0,T), \\
 \bu(\cdot,0) &= \bu_0(\cdot), &&\text{in }  \Omega,
 \end{align}
\end{subequations}
determines the reference solution for the data assimilation system, where $\bu$ is the velocity, $p$ is the pressure, $\nu > 0$ is viscosity,  $\Omega= \mathbb{T}^2 = \mathbb{R}^2 / \mathbb{Z}^2$ is the domain (and hence the domain has unit length $L=1$), $T>0$, $\bbf$ is some forcing, $\bu_0$ is the initial condition, and the system is equipped with periodic boundary conditions.  

The linear-nudging CDA algorithm applied to the 2D incompressible Navier-Stokes equations \eqref{NSEtrue} yields the system, with $\mu > 0$ a constant,
 \begin{subequations}\label{linDA}
 \begin{align}
  \bv_t + \bv\cdot \nabla \bv - \nu \triangle \bv + \nabla \widetilde{q} &= \bbf + \mu (I_h(\bu)-I_h(\bv)), &&\text{in }  \Omega\times (0,T), \\
  \nabla \cdot \bv &= 0, &&\text{in }  \Omega\times (0,T), \\
  \bv(\cdot,0) &= \bv_0(\cdot),&&\text{in }  \Omega.
  \end{align}
 \end{subequations}
 
 The nonlinear-nudging CDA algorithm applied to the 2D incompressible Navier-Stokes equations \eqref{NSEtrue} yields the system, with constants $\mu, \beta> 0$,
 \begin{subequations}\label{nonlinDA}
  \begin{align}
  \bv_t + \bv\cdot \nabla \bv - \nu \triangle \bv + \nabla q  &= \bbf  + \mu \mathcal{N}(I_h(\bu)-I_h(\bv)) 
  \\ \notag
  &\qquad + \beta (I_h(\bu)-I_h(\bv)),&&\text{in } \Omega\times (0,T), \\
   \nabla \cdot \bv &= 0, &&\text{in } \Omega\times (0,T),\\
  \bv(\cdot,0) &= \bv_0(\cdot), &&\text{in }  \Omega.
  \end{align}
 \end{subequations}


We recall the following well-known spaces. Let \[\mathcal{V} = \{\bu \in \dot{C}_{\text{p}}^\infty(\Omega) : \nabla\cdot\bu=0\},\] where $\dot{C}_{\text{p}}^\infty(\Omega)$ is the space of infinitely differentiable, mean-free, periodic functions on the torus.  We denote $H$ to be the closure of $\mathcal{V}$ in $L^2$ and $V$ to be the closure of $\mathcal{V}$ in $H^1$.  The inner-product on $H$ is the usual $L^2$ inner-product,
\begin{align*}
 (\bu,\bv):=\int_{\Omega}\sum\limits_{i,j=1}^2 u_i(\bx)v_j(\bx)d\bx, \quad \bu,\bv \in H,
\end{align*}
and we denote the inner-product on $V$ by
\begin{align*}
 ((\bu,\bv)):=\int_{\Omega}\sum\limits_{i,j=1}^2 \pd{u_j}{x_i}(\bx)\pd{v_j}{x_i}(\bx)d\bx, \quad \bu,\bv \in V.
\end{align*}
These yield the following norms.
\begin{align*}
 \|\bu\|_{H}:=\sqrt{(\bu,\bu)}, \qquad \|\bu\|_{V} := \sqrt{((\bu,\bu))}.
\end{align*}
Note that the definiteness of the V-norm follows from the Poincar\'e inequality (see, e.g., \cite{Constantin_Foias_1988}, \cite{Evans_2010}, \cite{Temam_2001_Th_Num}).

We denote the Leray projector as $P_\sigma : \dot{L}^2(\Omega) \to H$, where for smooth functions $\bw$,  $P_\sigma \bw = \bw - \nabla\triangle^{-1}\nabla\cdot\bw$, (the inverse Laplacian being computed with respect to periodic boundary conditions and the mean-free condition), and $P_\sigma$ is extended to $\dot{L}^2(\Omega)$ by continuity (see, e.g.,  \cite{Constantin_Foias_1988, Temam_2001_Th_Num}).  We denote $A:\mathcal{D}(A) \to H$ and $B: V\times V \to V^*$, where $A$ and $B$ are continuous extensions of the operators
\begin{align*}
 A\bu &= -P_\sigma\triangle\bu, \qquad \qquad \bu \in \mathcal{V}, \\
B(\bu,\bv) &= P_\sigma((\bu\cdot\nabla)\bv), \qquad \bu,\bv \in \mathcal{V},
\end{align*}
where $\mathcal{D}(A) := V \cap H^2(\Omega)$.  Note that $A$ is a linear, self-adjoint, and positive definite operator with compact inverse, so there exists an orthonormal basis of eigenfunctions $\set{\bw_i}_{i\in\nN}$ in $H$ such that $A\bw_i = \lambda_i\bw_i$, with eigenvalues $\lambda_i>0$ that are monotonically nondecreasing in $i$ (see, e.g., \cite{Constantin_Foias_1988, Temam_2001_Th_Num, Robinson_2001}).  Moreover, observe in this case (the case of periodic boundary conditions) that $A = -\triangle$. Furthermore, the following versions of Poincar{\'e} inequality hold,
\begin{align*}
\lambda_1 \|\bu\|_H^2 \leq \|\nabla\bu\|_H^2 \; \text{ for } \bu \in V,\\
\lambda_1 \|\nabla\bu\|_H^2 \leq \|A\bu\|_H^2 \; \text{ for } \bu \in \mathcal{D}(A),
\end{align*}
where $\lambda_1 = 4\pi^2$ is the first eigenvalue of the Stokes operator on $\Omega$.

Moreover, the bilinear operator, $B$, has the property that
\begin{align}
 \ip{B(\bu,\bv)}{\bw}_{V^*,V} = - \ip{B(\bu,\bw)}{\bv}_{V^*,V}
\end{align}
for all $\bu, \bv, \bw \in V$; below, we denote $\ip{\cdot}{\cdot} := \ip{\cdot}{\cdot}_{V^*,V}$.  This implies the following identity
\begin{align}
 \ip{B(\bu,\bv)}{\bv} = 0
\end{align}
for all $\bu, \bv \in V$.  Finally, the following standard inequalities hold (see, e.g., \cite{Constantin_Foias_1988,Foias_Manley_Rosa_Temam_2001,Robinson_2001,Temam_2001_Th_Num,Temam_1995_Fun_Anal})
\begin{align}
|\ip{B(\bu,\bv)}{\bw}| &\leq c \|\bu\|_H^{1/2}\|\bu\|_V^{1/2} \|\bv\|_V\|\bw\|_H^{1/2}\|\bw\|_V^{1/2}, \text{ for } \bu,\bv,\bw \in V, \label{BIN} \\
|\ip{B(\bu,\bv)}{\bw}| &\leq c \|\bu\|_H^{1/2}\|A\bu\|_H^{1/2} \|\bv\|_V\|\bw\|_H, \text{ for } \bu \in \mathcal{D}(A), \bv \in V, \bw \in H. \label{BIN3}
\end{align}

In this setting of 2D periodic boundary conditions, the following identity holds (see, e.g., \cite{Constantin_Foias_1988, Robinson_2001, Temam_2001_Th_Num})
\begin{align}\label{wwAw0}
 (B(\bw,\bw), A\bw) = 0 \text{ for every } \bw \in \mathcal{D}(A).
\end{align}
This implies the identity
\begin{align}\label{jacobi}
 (B(\bu,\bw),A\bw) + (B(\bw,\bu),A\bw) = - (B(\bw,\bw),A\bu) \text{ for all } \bu, \bw \in \mathcal{D}(A).
\end{align}

We assume that $\bbf \in L^\infty(0,T;H)$ so that $P_\sigma \bbf = \bbf$ (without loss of generality, since the gradient part of $\bbf$ can be absorbed into the pressure gradient).  Formally applying the Leray projection to \eqref{NSEtrue} yields the equivalent evolution system
\begin{align}\label{NSEtrue_leray}
 \bu_t + B(\bu,\bu) + \nu A\bu = \bbf,\\
 \bu(\bx,0) = \bu_0(\bx).
\end{align}

Similarly, the Leray projection can be formally applied to \eqref{linDA} to obtain the system
\begin{align}\label{linDA_Leray}
 \bv_t + B(\bv,\bv) + \nu A\bv &= \bbf + \mu P_\sigma(I_h(\bu)-I_h(\bv)),\\
 \bv(\bx,0) &= \bv_0(\bx).
\end{align}
and to \eqref{nonlinDA} to obtain the system
\begin{align}\label{nonlinDA_Leray}
 \bv_t + B(\bv,\bv) + \nu A\bv &= \bbf + \mu P_\sigma \mathcal{N}(I_h(\bu)-I_h(\bv)), 
 \\
   &\qquad + \beta P_\sigma(I_h(\bu)-I_h(\bv)),  \notag\\
 \bv(\bx,0) &= \bv_0(\bx).
\end{align}
where for all $\bw \in L^2(\Omega)$, we recall $\mathcal{N}(\bw) := \begin{cases} \|\bw\|_H^{-\gamma} \bw, & \|\bw\|_H \neq 0,\\
  0, & \|\bw\|_H= 0.
 \end{cases}$

 \noindent
 Note that, with this choice of $\mathcal{N}$, $\mu$ has units $(\text{length})^{2\gamma}/(\text{time})^{1+\gamma}$, whereas $\beta$ has units $1/(\text{time})$.

The pressure gradient can be recovered by employing the following corollary of de Rham's Theorem (see, e.g., \cite{Temam_2001_Th_Num,Foias_Manley_Rosa_Temam_2001,Wang_1993}):
\begin{align*}
 \bg = \nabla p \text{ with } p \text{ a distribution if and only if } \ip{\bg}{\bh} = 0 \text{ for all } \bh \in \mathcal{V}.
\end{align*}

Under this framework, we define the notion of a strong solution for the systems \eqref{NSEtrue_leray}, \eqref{linDA_Leray}, and \eqref{nonlinDA_Leray} (see, e.g., \cite{Constantin_Foias_1988, Robinson_2001, Temam_2001_Th_Num, Foias_Manley_Rosa_Temam_2001}).

\begin{definition}\label{strong_soln_def}
 A \textit{strong solution} of \eqref{NSEtrue_leray}, \eqref{linDA_Leray}, or \eqref{nonlinDA_Leray} is a function $\bu \in C([0,T];V) \cap L^2(0,T;\mathcal{D}(A))$ such that the equality in the system \eqref{NSEtrue_leray}, \eqref{linDA_Leray}, or \eqref{nonlinDA_Leray} is satisfied in $L^2(0,T;H)$, and its time derivative $\frac{d\bu}{dt} \in L^2(0,T;H)$.
\end{definition}

We cite the classical result of the existence of global strong solutions for \eqref{NSEtrue}

\begin{theorem}\label{strong_soln_NSE}
Given initial data $\bu_0 \in V$ and a forcing function $\bbf \in L^2(0,T;H)$, there exists a unique strong solution to \eqref{NSEtrue_leray} $\bu$ such that $\bu \in C([0,T];V) \cap L^2(0,T;\mathcal{D}(A))$ and $\frac{d\bu}{dt} \in L^2(0,T;H)$.
\end{theorem}
 
For \eqref{NSEtrue_leray}, we denote the dimensionless Grashof number as
\begin{align}
G &= \frac{1}{\lambda_1\nu^2} \limsup\limits_{t \to \infty} \|\mathbf{f}(t)\|_H \label {G}.
\end{align}

The following theorem, see, e.g., \cite{Robinson_2001, Temam_2001_Th_Num, Constantin_Foias_1988, Foias_Manley_Rosa_Temam_2001, Dascaliuc_Foias_Jolly_2010}, details results on the large-time behavior of the strong solutions to \eqref{NSEtrue}.

\begin{theorem}\label{ubounds}
Fix $T >0$. Suppose that $\bu$ is a strong solution of $\eqref{NSEtrue}$, corresponding to the initial data $\bu_0 \in V$. Then there exists a time $t_0 \geq 0$ which depends on $\|\bu_0\|_H$ such that for all $t \geq t_0$,
\begin{align}
\|\bu(t)\|_H^2 \leq 2G^2\nu^2  \; \text{ and } \int\limits_t^{t+T} \|\bu(\tau)\|_V^2 d\tau \leq 2\left(1+T\lambda_1\nu\right) G^2\nu,
\end{align}
moreover 
\begin{align}
\| \bu(t)\|_V^2 \leq 2\lambda_1G^2\nu^2, \qquad \int\limits_t^{t+T} \|A\bu(\tau)\|_H^2 d\tau \leq 2\left(1+T\lambda_1\nu\right) \lambda_1G^2\nu.
\end{align}
Furthermore, if $\mathbf{f} \in H$ is time-independent then
\begin{align}\label{Au L2 bound}
\|A\bu(t)\|_H^2 \leq \lambda_1^2\nu^2c (1+G)^4.
\end{align}
\end{theorem}

We assume throughout the present work that the  operator $I_h$ is a linear operator satisfying the following conditions
\begin{subequations}\label{Ih_always_assumptions}
\begin{align}
\label{Ih_interp_bound_1}
	\|\bphi-I_h(\bphi)\|_{L^2(\Omega)}^2 &\leq c_0h^2\|\nabla\bphi\|_{L^2(\Omega)}^2, \text{ for all } \bphi \in \dot{H}^1(\Omega),\\
  \label{Ih_mean_free}
\int_\Omega I_h(\bphi) &= 0\text{ whenever }\int_\Omega \bphi = 0.
\end{align}
\end{subequations}
In our theorems, we make various additional assumptions about the interpolant $I_h$.   We record these here for reference, though we note that some of our theorems hypothesize only a subset of these assumptions.
\begin{subequations}
\begin{align}
  \label{Ih_projection}
 I_h^2 &= I_h,\\
 \label{Ih_symmetric}
     (I_h(\bphi), \bpsi) &= (\bphi, I_h(\bpsi)) \text{ for all }\bphi, \bpsi \in \dot{L}^2(\Omega),\\
\label{Ih_positive}
 (I_h\bphi,\bphi) &\geq 0 \text{ for all } \bphi \in \dot{L}^2(\Omega),\\
 \label{Ih_bounded}
 \|I_h\bphi\|_{L^2(\Omega)} &\leq \alpha\|\bphi\|_{L^2(\Omega)}  \text{ for some }\alpha>0 \text{ and for all } \bphi \in \dot{L}^2(\Omega),\\
 \label{Ih_A_positive}
 (I_h\bphi,A\bphi) &\geq 0 \text{ and for all } \bphi \in \mathcal{D}(A).
 \end{align}  
\end{subequations}
Note that Fourier truncation and local averaging over finite volume elements are both operators that
satisfy \eqref{Ih_always_assumptions} (see, e.g., \cite{Azouani_Olson_Titi_2014, Cockburn_Jones_Titi_1997, Jones_Titi_1993, Brenner_Scott_FE_2008}).  There seems to be a technical constraint on allowing more general interpolants, cf. Remark \ref{type2_interp_rmk}.
Since we are working in a mean-free space Poincar{\'e}'s inequality applies, and   combined with \eqref{Ih_always_assumptions} we have the following bound on $\|I_h(\bphi)\|_{L^2(\Omega)}$:
\begin{align}\label{interp_bound_2}
 \|I_h(\bphi)\|_{L^2(\Omega)} &\leq \|\bphi-I_h(\bphi)\|_{L^2(\Omega)} + \|\bphi\|_{L^2(\Omega)}, \notag\\
 &\leq \sqrt{c_0}h\|\nabla\bphi\|_{L^2(\Omega)} + \lambda_1^{-1/2} \|\nabla\bphi\|_{L^2(\Omega)}, \notag\\
 &= (\sqrt{c_0}h + \lambda_1^{-1/2}) \|\nabla\bphi\|_{L^2(\Omega)}.
\end{align}
For uniqueness of solutions to the nonlinear-nudging system and the convergence of the solutions of the nonlinear-nudging system to the unknown reference solution of \eqref{NSEtrue}, we will make various assumptions on the linear interpolant $I_h$, namely \eqref{Ih_positive} and \eqref{Ih_bounded} above. 

For reference, we state the existence and convergence theorems of the linear-nudging CDA algorithm as proved in \cite{Azouani_Olson_Titi_2014}.

\begin{theorem}\cite{Azouani_Olson_Titi_2014}
 Suppose $I_h$ satisfies \eqref{Ih_always_assumptions}, $\bv_0 \in V$, and $\mu c_0h^2 \leq \nu$, where $c_0$ is the constant in \eqref{Ih_always_assumptions}.  Then system \eqref{linDA_Leray} has a unique strong solution such that \[\bv \in C([0,T];V) \cap L^2(0,T;\mathcal{D}(A)) \text{ and } \frac{d\bv}{dt} \in L^2(0,T;H),\]
 for any $T> 0$.  Furthermore, this solution depends continuously on the initial data $\bv_0$ in the $V$ norm.
\end{theorem}

\begin{theorem}\label{CDA_Hnorm}\cite{Azouani_Olson_Titi_2014}
 Let $\Omega = \mathbb{T}^2$ and let $\bu$ be a solution to \eqref{NSEtrue_leray} with periodic boundary conditions.  Let $I_h$ satisfy \eqref{Ih_always_assumptions}.  Then provided $\mu \geq 5\lambda_1c^2\nu G^2$ and $h \leq \left( \frac{1}{10\lambda_1c^2c_0G^2}\right)^{1/2}$, for any $\bv_0 \in V$ the solution to \eqref{linDA_Leray} converges to the solution of \eqref{NSEtrue_leray} in $H$ exponentially fast as $t \to \infty$, with exponential rate $\mu/2$.
\end{theorem}

\begin{theorem}\label{CDA_Vnorm}\cite{Azouani_Olson_Titi_2014}
 Let $\Omega = \mathbb{T}^2$ and let $\bu$ be a solution to \eqref{NSEtrue_leray} with periodic boundary conditions.  Let $I_h$ satisfy \eqref{Ih_always_assumptions}.
 Then provided $\mu c_0h^2 \leq \nu$ and $\mu \geq 3\nu\lambda_1(2c\log(2c^{3/2} + 8c\log(1+G))G$ where $c$ is a constant dependent on the size of the domain, for any $\bv_0 \in V$ the solution to \eqref{linDA_Leray} converges to the solution of \eqref{NSEtrue_leray} in $V$ exponentially fast as $t \to \infty$, with exponential rate $\mu/2$.
\end{theorem}


%
%
%

%
%

We will also employ the following elementary lemma, the proof of which is in the Appendix.



\begin{lemma}\label{small_h}
 Fix $\gamma\in(0,1)$ and $\epsilon >0$.  Given $a > 0$ and  \[\delta := \min\left\{ a/2, a^{\frac{2-\gamma}{2}}\lp \frac{\epsilon}{\lp\frac{2-\gamma}{2}\rp^{\frac{2-\gamma}{\gamma}} - \lp\frac{2-\gamma}{2}\rp^{\frac{2}{\gamma}}} \rp^{\frac{\gamma}{2}}\right\},\] then the function $f:[0,\infty)\mapsto\nR$ defined by \[f(x) = ax^2 - \delta x^{2-\gamma},\] satisfies $f(x) \geq -\epsilon$ for all $x\geq0$.
\end{lemma}

\section{Global Existence and Uniqueness of the Nonlinear-Nudging System}\label{secExis}
\noindent
Before we can prove convergence of the solutions of \eqref{nonlinDA} to the reference solution $\bu$ of \eqref{NSEtrue}, we must demonstrate that solutions to \eqref{nonlinDA} exist globally in time; we employ fixed point methods.  The advantage of this approach is many-fold, harnessing the properties of solutions to the Navier-Stokes equation directly and using the monotonicity of the nonlinear-nudging term to best effect.  


\begin{remark}
    Note that in the theorems, in the following sections, we claim that we critically assume $0 < \mu c_0 h^2 < \nu$, and we remark there is a constant $1$ multiplying $\mu$ to maintain proper units.
\end{remark}

\begin{theorem}\label{thm_existence}
 Let $T> 0$. Suppose $I_h$ satisfies \eqref{Ih_always_assumptions}, $0 < \mu c_0h^2 < \nu$, and $0 < \beta c_0h^2 < \nu$.  Fix $0 < \gamma < 1$.  Let $\bv_0 \in V$ be initial data and $\bbf \in H$ be the time-independent forcing function for \eqref{nonlinDA_Leray}.  Let $\bu \in C([0,T];V) \cap L^2(0,T;\mathcal{D}(A))$ be a strong solution to system \eqref{NSEtrue_leray} with initial data $\bu_0 \in V$ and the same forcing function $\bbf$.  Then there exists a strong solution $\bv \in C([0,T];V) \cap L^2(0,T;\mathcal{D}(A))$ to system \eqref{nonlinDA_Leray}.   
\end{theorem}

\begin{proof}

First, we show that \[g(\bv) = \bbf + \mu P_\sigma\mathcal{N}(I_h(\bu-\bv)) + \beta P_\sigma I_h(\bu-\bv)\] maps elements $\bv \in L^2(0,T;V)$ to $L^2(0,T;H)$.  By assumption $\bbf \in H$ and hence trivially $\int_0^T |\bbf|^2 dx = T|\bbf|^2 < \infty$, so we only need to prove that the second and third terms reside in $L^2(0,T;H)$:

\begin{align}\label{NL_IH_bound}
&\quad
 \int_0^T \|P_\sigma\mathcal{N}(I_h(\bu-\bv))\|_H^2 dt 
 \\
 &\leq \int_0^T \|\mathcal{N}(I_h(\bu-\bv))\|_{L^2(\Omega)}^2 dt 
 = \int_0^T \|I_h(\bu-\bv)\|_{L^2(\Omega)}^{2(1-\gamma)} dt 
 \notag 
 \\
 &\leq (\lambda_1^{-1/2} + h\sqrt{c_0})^{2(1-\gamma)} \int_0^T \|\bu-\bv\|_V^{2(1-\gamma)} dt 
 \notag
 \\
 &\leq 2^{2(1-\gamma)}(\lambda_1^{-1/2} + h\sqrt{c_0})^{2(1-\gamma)} \int_0^T \lp \|\bu\|_V^{2(1-\gamma)}+\|\bv\|_V^{2(1-\gamma)} \rp dt
  \notag
  \\
 &\leq 2^{2(1-\gamma)}(\lambda_1^{-1/2} + h\sqrt{c_0})^{2(1-\gamma)} T^{\gamma}\lp\|\bu\|_{L^2(0,T;V)}^{2(1-\gamma)} + \|\bv\|_{L^2(0,T;V)}^{2(1-\gamma)}\rp
 \notag  
 \\
 &< \infty,\notag
\end{align}
where for the second inequality, we applied \eqref{interp_bound_2}.  The same analysis with $\gamma = 0$ shows that the third term is also in $L^2(0,T; H)$.  Thus, $g(\bv) \in L^2(0,T; H)$ for any $\bv \in L^2(0,T;V)$.
Next, we show $g: L^2(0,T;V)\to L^2(0,T;H)$ is continuous.  To this end, let $\bv_k\to \bv$ in $L^2(0,T;V)$.
For the sake of contradiction suppose there exists an $\epsilon_0 > 0$ and a subsequence $\{\bv_{k_j}\}$ such that $\|g(\bv_{k_j}) - g(\bv)\|_{L^2(0,T;H)} \geq \epsilon_0$.  Since $\|\bv_{k_j}-\bv\|_{L^2(0,T;V)} \to 0$, there exists a subsequence $\{\bv_{k_{j_l}}\}$  such that $\|\bv_{k_{j_l}}(t) - \bv(t)\|_V \to 0$ for almost every $t \in [0,T]$.
This implies that $\|I_h(\bu-\bv_{k_{j_l}}) - I_h(\bu-\bv)\|_H \to 0$ pointwise a.e. in time by \eqref{interp_bound_2}, 
 so that $\|g(\bv_{k_{j_l}}) - g(\bv)\|_H^2 \to 0$ pointwise a.e. in time.  Furthermore, there exists $\bw \in L^2(0,T;V)$ such that $\|I_h(\bu-\bv_{k_{j_l}})\|_H \leq \|\bw\|_{V}$ for all $k_{j_l} \in \mathbb{N}$ and for a.e. $t\in [0,T]$; then 
 \begin{align*}
 	&\quad
 	\int_0^T \|g(\bv_{k_{j_l}})- g(\bv)\|_H^2 dt 
 \\
    &\leq  
 	\int_0^T \Big(\mu\| I_h(\bu-\bv_{k_{j_l}})\|_H^{1-\gamma} + \mu \|I_h(\bu-\bv)\|_H^{1-\gamma} 
 \\
 	&\qquad \qquad + \beta \|I_h(\bv_{k_{j_l}} - \bu)\|_H + \beta \|I_h(\bv - \bu)\|_H \Big)^2 dt 
  \\
 	&\leq 
    \int_0^T \lp \mu \|\bw\|_V^{2(1-\gamma)} + \beta \|\bw\|_V^2 + \mu \|I_h(\bu-\bv)\|_H^{1-\gamma} + \beta \|I_h(\bu-\bv)\|_H \rp^2 
 	< \infty,
 \end{align*}
 where the finiteness follows directly from the bounds computed in \eqref{NL_IH_bound}.
 Hence, by the Lebesgue Dominated Convergence Theorem,  
  $g(\bv_{k_{j_l}}) \to g(\bv)$ strongly in $L^2(0,T;H)$.   This is a contradiction, and hence $g: L^2(0,T;V) \to L^2(0,T;H)$ is continuous.

Given $\bv \in L^2(0,T;V)$, consider the system 
\begin{align}\label{nonlinDA_alt}
 \widetilde{\bu}_t + B(\widetilde{\bu},\widetilde{\bu}) + \nu A\widetilde{\bu} &= g(\bv), \\
 \widetilde{\bu}(\bx,0) &= \bv_0(\bx).
\end{align}
Let $F(\bv) = \widetilde{\bu} \in C([0,T];V)\cap L^2(0,T;\mathcal{D}(A))$ be the unique strong solution to \eqref{nonlinDA_alt}, guaranteed by Theorem \ref{strong_soln_NSE} since $g(\bv) \in L^2(0,T;H)$ (hence $F$ is well-defined).  

To show that $F$ is continuous, take a sequence $\{\bv_k\} \subset L^2(0,T;V)$ with $\bv_k  \to \bv$ in $L^2(0,T;V)$, 
with the associated sequence of solutions $F(\bv_k) = \tbu_k$ to \eqref{nonlinDA_alt}. Let $M > 0$ be such that $\|\bv_k\|_{L^2(0,T;V)} \leq M$ for all $k \in \mathbb{N}$.   
Take the $H$ inner-product of \eqref{nonlinDA_alt} with $A\tbu_k$ and use standard energy-estimate techniques to obtain
\begin{align*}
\frac{1}{2} \frac{d}{dt} \|\tbu_k\|_V^2 + \nu \|A\tbu_k\|_{H}^2 &\leq \frac{1}{2\nu}\|g(\bv_k)\|_H^2 + \frac{\nu}{2} \|A\tbu_k\|_{H}^2 
\end{align*}
hence
\begin{align*}
\frac{d}{dt} \|\tbu_k\|_V^2 + \nu \|A\tbu_k\|_H^2 &\leq \frac{1}{\nu}\|g(\bv_k)\|_H^2,
\end{align*}
which implies that, since the initial data is independent of $k$, 
\begin{align}\label{estimate1}
&\quad
 \|\tbu_k(t)\|_V^2 + \nu \int_0^t \|A\tbu_k\|_H^2 dt 
 \leq \|\tbu_k(0)\|_V^2 + \frac{1}{\nu}\int_0^t \|g(\bv_k)\|_{H}^2 dt  
\\
 &\leq \|\bv_0\|_V^2 + \frac{2}{\nu} T\|f\|_H^2 + \frac{C_\gamma(\mu,h)}{\nu}T^{\gamma}(\|\bu\|_{L^2(0,T;V)}^{2(1-\gamma)} + \|\bv_k\|_{L^2(0,T;V)}^{2(1-\gamma)}) 
 \nonumber
\\
&\qquad \qquad
 + \frac{C(\beta,h)}{\nu}(\|\bu\|_{L^2(0,T;V)}^{2} + \|\bv_k\|_{L^2(0,T;V)}^{2})
 \nonumber
\\
 &\leq 
 \|\bv_0\|_V^2 + \frac{2}{\nu} T\|f\|_H^2 + \frac{C_\gamma(h,\mu)}{\nu}T^{\gamma}(\|\bu\|_{L^2(0,T;V)}^{2(1-\gamma)}  + M^{2(1-\gamma)})\nonumber
 \\
 &\qquad \qquad 
 + \frac{C(\beta,h)}{\nu}(\|\bu\|_{L^2(0,T;V)}^{2} + M^{2}),\nonumber
\end{align}
where $C_\gamma(h,\mu) = 2^{2-\gamma}\mu^2\lp\sqrt{c_0}h+\lambda_1^{-\frac12}\rp^{2(1-\gamma)}$ and $C(h,\beta) = 4\beta^2 \lp\sqrt{c_0}h+\lambda_1^{-\frac12}\rp^2$.

By \eqref{estimate1}, we know that $\{\tilde{\bu}_k\}$ is uniformly bounded in $C([0,T];V) \cap L^2(0,T;\mathcal{D}(A))$.  We show that $\{\tilde{\bu}_k\}$ is a Cauchy sequence.  Let $m \geq k$ and $\bw_{m,k} = \tilde{\bu}_m - \tilde{\bu}_k$; then $\bw_{m,k}$ satisfies
\begin{align}
    \frac{d}{dt}\bw_{m,k} + \nu A \bw_{m,k} + B(\tilde{\bu}_k, \bw_{m,k}) + B(\bw_{m,k}, \tilde{\bu}_k) + B(\bw_{m,k}, \bw_{m,k}) = g(\bv_k) - g(\bv_m).
\end{align}
Taking the inner product with $A\bw_{m,k}$ in $H$ and applying \eqref{wwAw0}, \eqref{jacobi},
\begin{align}
    \frac12 \frac{d}{dt} \|\bw_{m,k}\|_V^2 + \nu\|A\bw_{m,k}\|_H^2 = (B(\bw_{m,k},\bw_{m,k}), A\tilde{\bu}_k) + (g(\bv_k) - g(\bv_m), A \bw_{m,k}).
\end{align}
Thus, employing Poincar{\'e}'s and Young's inequalities,
\begin{align}
    \frac12 \frac{d}{dt} \|\bw_{m,k}\|_V^2 + \nu\|A\bw_{m,k}\|_H^2 
    &\leq 
    c \|\bw_{m,k}\|_H^\frac12\|A\bw_{m,k}\|_H^\frac12\|\bw_{m,k}\|_V\|A\tilde{\bu}_k\|_H
    \\\notag
    &\qquad + \|g(\bv_k) - g(\bv_m)\|_H\|A \bw_{m,k}\|_H
    \\\notag
    &\leq 
    c^{\frac43}\lambda_1^{-\frac13}\nu^{-3}\|A\tilde{\bu}_k\|_H^\frac43\|\bw_{m,k}\|_V^2 
    \\\notag
    &\qquad + \frac{\nu}2\|A\bw_{m,k}\|_H^2 + \frac{1}{\nu}\|g(\bv_k) - g(\bv_m)\|_H^2,
\end{align}
which yields
\begin{align}\label{Fcontinuity_energyBds}
    \frac{d}{dt}\|\bw_{m,k}\|_V^2 + \nu\|A\bw_{m,k}\|_H^2 
    \leq 
    2c^{\frac43}\lambda_1^{-\frac13}\nu^{-3}\|A\tilde{\bu}_k\|_H^\frac43\|\bw_{m,k}\|_V^2 
    + \frac{2}{\nu}\|g(\bv_k) - g(\bv_m)\|_H^2
\end{align}
which implies, with $\tilde{c} = 2c^{\frac43}\lambda_1^{-\frac13}\nu^{-3}$
\begin{align}
    \frac{d}{dt}\|\bw_{m,k}\|_V^2 
    \leq 
    \tilde{c} \|A\tilde{\bu}_K\|_H^\frac43\|\bw_{m,k}\|_V^2 + \frac{2}{\nu}\|g(\bv_k) - g(\bv_m)\|_H^2
\end{align}
and by Gr{\"o}nwall's inequality, (since $\bw_{k,m}(0) = 0$),
\begin{align}
    \|\bw_{k,m}(t)\|_V^2 
    &\leq
    \int_0^t \text{exp}
    \lp  \tilde{c}\int_s^t\|A\tilde{\bu}_k(\tau)\|_H^\frac43 d\tau
    \rp
    \frac{2}{\nu}\|g(\bv_k(s)) - g(\bv_m(s))\|_H^2 ds
    \\
    &\leq 
    \text{exp}
    \lp  \tilde{c}\int_0^T\|A\tilde{\bu}_k(\tau)\|_H^\frac43 d\tau
    \rp
    \int_0^T
    \frac{2}{\nu}\|g(\bv_k(s)) - g(\bv_m(s))\|_H^2 ds
\end{align}
Thus, taking the supremum and applying H{\"o}lder's inequality,
\begin{align}
\sup\limits_{0\leq t\leq T} \|\bw_{k,m}(t)\|_V^2
& \leq
\text{exp}\lp\tilde{c}\lp \int_0^T \|A\tilde{\bu}_k(\tau)\|_H^2 d\tau\rp^\frac23 T^\frac13\rp\|g(\bv_k(s)) - g(\bv_m(s))\|_{L^2(0,T;H)}^2.
\end{align}
By the uniform bound on $\|A\tilde{\bu}_k(\tau)\|_{L^2(0,T;H)}^2$ and since $g$ is continuous, we have the $\{\tilde{\bu}_k\}$ is Cauchy in the $C([0,T];V)$ norm hence it converges to some $\tbu\in C([0,T];V)$.
Next, we directly integrate \eqref{Fcontinuity_energyBds}, using $\bw_{k,m}(0) = 0$, to obtain
\begin{align}
\nu\int_0^T\|A\bw_{m,k}\|_H^2 ds
    \leq
    \tilde{c} &\lp\int_0^T\|A\tilde{\bu}_k(s)\|_H^2 ds\rp^\frac23 T^\frac13\lp\sup\limits_{0\leq t\leq T}\|\bw_{m,k}\|_V^2 \rp
    \\\notag
    &+ \frac{2}{\nu}\|g(\bv_k) - g(\bv_m)\|_{L^2(0,T;H)}^2,
\end{align}
which implies that $\{A\tilde{\bu}_k\}$ is Cauchy in $L^2(0,T;H)$ and converges to $A\tilde{\bu}$ in this norm.
Next, we show that $\frac{d\tilde{\bu}_k}{dt}\to \frac{d\tilde{\bu}}{dt}$ in $L^2(0,T;H)$.  First we observe that, via \eqref{BIN3}, 
\begin{align}
\notag
    &\|B(\tilde{\bu}_k, \tilde{\bu}_k) - B(\tilde{\bu}, \tilde{\bu})\|_H
    + \|B(\tilde{\bu}_k, \tilde{\bu})\|_H 
    + \|B(\tilde{\bu}_k, \tilde{\bu}_k - \tilde{\bu})\|_H 
    \\
    \leq&
    c \|\tilde{\bu}_k-\tilde{\bu}\|_H^\frac12\|A\tilde{\bu}_k-A\tilde{\bu}\|_H^\frac12\|\tilde{\bu}\|_V 
    + c \|\tilde{\bu}_k\|_H^\frac12\|A\tilde{\bu}_k\|_H^\frac12\|\tilde{\bu}_k-\tilde{\bu}\|_V.
\end{align}
And thus,
 \begin{align*}
  \int_0^T \|B(\tbu_k-\tbu,\tbu_k-\tbu)\|_H^2 dt 
  &\leq 
  c \frac{1}{\lambda_1}\|\tbu\|_{L^\infty(0,T;V)}^2\|A\tbu_k-A\tbu\|_{L^2(0,T;H)}^2 
  \\
  &\qquad + c \|\tbu_k-\tbu\|_{L^\infty(0,T;V)}^2\|\tbu_k\|_{L^2(0,T;H)}\|A\tbu_k\|_{L^2(0,T;H)},
 \end{align*}
 which from the fact that $\tbu_k\to \tbu$ strongly in $C([0,T];V)\cap L^2(0,T;\mathcal{D}(A))$ implies
$B(\tilde{\bu}_k,\tilde{\bu}_k)$ converges to $B(\tilde{\bu},\tilde{\bu})$ in $L^2(0,T;H)$.  Since 
\begin{align}\label{derivativeKbound}
    \frac{d\tilde{\bu}_k}{dt} 
    &= -\nu A\tilde{\bu}_k - B(\tilde{\bu}_k, \tilde{\bu}_k) + g(\bv_k),
\end{align}
the right-hand side converges in $L^2(0,T;H)$ thus $\frac{d\tilde{\bu}_k}{dt}$ converges to $\frac{d\tilde{\bu}}{dt}$ in $L^2(0,T;H)$.  Thus, $F(\bv_k) \to F(\bv)$ in $C([0,T];V)\cap L^2(0,T;\mathcal{D}(A))$.


Next we show that 
$F$ is a compact operator on $L^2(0,T;V)$. For any bounded sequence $\{\bv_k\} \subset L^2(0,T;V)$ the estimate \eqref{estimate1} holds, implying that $\{\tbu_k\}$ is uniformly bounded, in particular, in $L^2(0,T;\mathcal{D}(A))$, and the same arguments can be followed  \eqref{derivativeKbound} show uniform boundedness instead of convergence, i.e. $\{\frac{d\tbu_k}{dt}\}$ is uniformly bounded in $L^2(0,T;H)$.  Thus, by Aubin's Compactness Theorem, there exists a subsequence $\{\bv_{k_j}\}$ such that $F(\bv_{k_j}) = \tbu_{k_j}$ converges strongly in $L^2(0,T;V)$. Thus, $F$ is a (nonlinear) continuous compact operator.

We implement a version of the Schauder Fixed Point Theorem which states that for a closed, bounded, convex set B in a Banach space $X$, if $F: X \to X$ is a compact operator such that $F:B \to B$, then $F$ has a fixed point in $B$, (see, e.g., \cite{Conway_1990}). For given initial data $\bv_0\in V$, fix $R > \frac{\|\bv_0\|_V^2}{\nu\lambda_1}$. Set\footnote{It is straightforward, though slightly laborious, to check that the expression for $T_*$ is dimensionally correct.} 
\begin{align*}
    T_* := \frac{R}
    {
    \nu\lambda_1 R
    + \frac{1}{\nu^2\lambda_1}\|\bbf\|_H^2 
    + \frac{C_\gamma(h,\mu)}{\nu^{1+\gamma}\lambda_1^\gamma}
    \lp\|\bu\|_{L^2(0,T;V)}^{2(1-\gamma)} + R^{(1-\gamma)}\rp
     + \frac{C(h,\beta)}{\nu}
     \lp\|\bu\|_{L^2(0,T;V)}^2 + R\rp}.
\end{align*}
and
\begin{align*}
    B := \left\{\bv \in L^2(0,T;V) :  \int_0^{T_*} \|\bv\|_V^2 \leq R\right\},
\end{align*}
Notice that $T_* < \frac{1}{\nu\lambda_1}$.
Given any $\bv \in B$, we note that by definition $F(\bv) = \tbu \in C([0,T_*];V) \cap L^2(0,T_*;\mathcal{D}(A))$ and thus $\tbu \in L^2(0,T_*;V)$.
Moreover, using an identical estimate to the first inequality in \eqref{estimate1}, except that we integrate over $[0,T_*]$ and use $g(\bv)$ instead of $g(\bv_k)$, we 
obtain
\begin{align*}
    \int_0^{T_*} \|\tbu\|_V^2 dt
    &\leq
    T_*\|\bv_0\|_V^2 
    + \frac{1}{\nu} T_*^2\|\bbf\|_H^2
\\
    &\qquad \qquad 
     + T_*^{1+\gamma}\frac{C_\gamma(h,\mu)}{\nu}\lp\lp\int_0^{T_*} \|\bu\|_V^2\rp^{1-\gamma}  + \lp\int_0^{T_*} \|\bv\|_V^2\rp^{1-\gamma} \rp
 \\
    &\qquad \qquad 
    + T_*\frac{C(\beta,h)}{\nu}\lp\int_0^{T_*} \|\bu\|_V^2 + \int_0^{T_*} \|\bv\|_V^2\rp 
 \\
    &\leq 
    T_* \Big( \nu\lambda_1 R + \frac{1}{\nu} T_* \|\bbf\|_H^2
    + \frac{C_\gamma(h,\mu)}{\nu}T_*^\gamma \lp \|\bu\|_{L^2(0,T;V)}^{2(1-\gamma)} + R^{1-\gamma}\rp
\\
    &\qquad \qquad
    + \frac{C(\beta,h)}{\nu}\lp\|\bu\|_{L^2(0,T;V)}^{2} + R\rp\Big)
\\
    &\leq 
    T_* \Big( \nu\lambda_1 R + \frac{1}{\nu^2\lambda_1} \|\bbf\|_H^2
    + \frac{C_\gamma(h,\mu)}{\nu^{1+\gamma}\lambda_1^\gamma} \lp \|\bu\|_{L^2(0,T;V)}^{2(1-\gamma)} + R^{1-\gamma}\rp
\\
    &\qquad \qquad
    + \frac{C(\beta,h)}{\nu}\lp\|\bu\|_{L^2(0,T;V)}^{2} + R\rp \Big)\leq R
\end{align*}
by the definition of $T_*$.  
In other words, 
$\widetilde{\bu} \in B$, and hence, $F:B\to B$.  Since $F$ is compact on $L^2(0,T;V)$, there exists a fixed point of $F$ in $B$, i.e., $F(\bv) = \bv$ on $[0,T_*]$.  Call this fixed point $\bv_1 \in C([0,T];V) \cap L^2(0,T;\mathcal{D}(A))$.  Consider \eqref{nonlinDA_alt} with initial data $\bv_1(T_*/2)$ (which is allowed, because solutions must be continuous in time).  Now choose $\tilde{R} > \max\left\{R,  \frac{\|\bv_1(T_*/2)\|_V^2}{\nu\lambda_1}\right\}$ and 
\begin{align}
    \tilde{B}:= \left\{\bv \in L^2(0,T;V)  : \int_{T_*/2}^{\tilde{T}_*} \|\bv\|_V^2 \leq \tilde{R}\right\}
\end{align}
and
\begin{align*}
    \tilde{T}&_* := \frac{T_*}{2} + \Delta T,
\end{align*}
where 
\begin{align*}
    \Delta T = \frac{\tilde{R}}
    {
    \nu\lambda_1 \tilde{R}
    + \frac{1}{\nu^2\lambda_1}\|\bbf\|_H^2 
    + \frac{C_\gamma(h,\mu)}{\nu^{1+\gamma}\lambda_1^\gamma}
    \lp\|\bu\|_{L^2(0,T;V)}^{2(1-\gamma)} + \tilde{R}^{(1-\gamma)}\rp
     + \frac{C(h,\beta)}{\nu}
     \lp\|\bu\|_{L^2(0,T;V)}^2 + \tilde{R}\rp}.
\end{align*}
Notice that $T_*/2 < T_* < \tilde{T}_*$, and notice moreover that the length $\Delta T$ of the interval is slightly larger than that of the previous interval because the function $\frac{x}{a+bx^{1-\gamma}+cx}$ is monotonically increasing in $x$ for $a,b,c,x>0$.  Integrating from $T_*/2$ to $\tilde{T}_*$ yields that $\|\tbu\|_{L^2(T_*/2,\tilde{T}_*;V)}^2 \leq \tilde{R}$ implying that $\tbu \in B$.  Hence, $F$ has a fixed point $\bv_2 \in \tilde{B}$ such that $F(\bv_2) = \bv_2$ on $[T_*/2,\tilde{T}_*]$, and moreover $\bv_2(T_*/2) = \bv_1(T_*/2)$.  By induction on $\tilde{T}_*$, $F$ has a fixed point $\bv \in C([0,T];V) \cap L^2(0,T; \mathcal{D}(A))$.
\end{proof}

\begin{remark}
Alternatively one can use Schaefer's Fixed Point Theorem in which one does not have to bootstrap in time as in the above proof.
\end{remark}

\begin{theorem}\label{uniqueness}
 Suppose $I_h$ 
 satisfies \eqref{Ih_always_assumptions}, \eqref{Ih_projection}, and \eqref{Ih_symmetric}.  Then system \eqref{nonlinDA_Leray} with initial data $\bv_0 \in V$ possesses a unique strong solution.
\end{theorem}
\begin{proof}
 Suppose $\bv_1$ and $\bv_2$ are two strong solutions to \eqref{nonlinDA} with the same initial condition.  Let $\bw_1 = \bu-\bv_1$, $\bw_2 = \bu-\bv_2$, and $\bV := \bv_1-\bv_2 = \bw_2-\bw_1$. Then $\bV$ solves the system
 \begin{align*}
  \frac{d\bV}{dt} + \nu A \bV + B(\bV,\bv_1) + B(\bv_2,\bV) &= \mu P_\sigma \left(\mathcal{N}(I_h(\bw_1)) - \mathcal{N}(I_h(\bw_2))\right) \\
  &\quad - \beta I_h(\bV)
  \\
  \nabla \cdot \bV &= 0 
  \\
  \bV(\bx,0) &= 0.
 \end{align*}
We prove that $\mathcal{N}$ is a monotone operator on $L^2(\Omega)$: given $\bu_1, \bu_2 \in L^2(\Omega)$, with $\bu_1 \neq \bu_2$ and non-zero (the proof is similar if for instance $\bu_2=\mathbf{0}$),
 \begin{align*}
 &\quad
  (\|\bu_1\|_H^{-\gamma} \bu_1 - \|\bu_2\|_H^{-\gamma} \bu_2, \bu_1-\bu_2) 
  \\
  &= \|\bu_1\|_H^{2-\gamma} + \|\bu_2\|_H^{2-\gamma} - (\|\bu_1\|_H^{-\gamma}+\|\bu_2\|_H^{-\gamma})(\bu_1,\bu_2) \\
  &\geq \|\bu_1\|_H^{2-\gamma} + \|\bu_2\|_H^{2-\gamma} - (\|\bu_1\|_H^{-\gamma}+\|\bu_2\|_H^{-\gamma})\|\bu_1\|_H\|\bu_2\|_H \\
  &= \|\bu_1\|_H^{2-\gamma} - \|\bu_1\|_H^{1-\gamma}\|\bu_2\|_H + \|\bu_2\|_H^{2-\gamma} - \|\bu_2\|_H^{1-\gamma}\|\bu_1\|_H \\
  &= (\|\bu_1\|_H^{1-\gamma} - \|\bu_2\|_H^{1-\gamma})(\|\bu_1\|_H-\|\bu_2\|_H) \\
  &\geq 0.
 \end{align*}
 
 Since $I_h$ is linear and satisfies \eqref{Ih_symmetric}, we take the inner-product with $\bV$ 
 \begin{align}
 &\quad
  \frac{1}{2}\frac{d}{dt} \|\bV\|_H^2 + \nu \|\bV\|_V^2 
  \\ \notag
  &= -(B(\bV,\bv_1),\bV) + \mu (P_\sigma(\mathcal{N}(I_h(\bw_1)) - \mathcal{N}(I_h(\bw_2))),\bV) - \beta(I_h(\bV),\bV)
  \\ \notag
    &=-(B(\bV,\bv_1),\bV) - \mu \lp \frac{I_h(\bw_1)}{\|I_h(\bw_1)\|_{L^2(\Omega)}} - \frac{I_h(\bw_2)}{\|I_h(\bw_2)\|_{L^2(\Omega)}},\bw_1 - \bw_2\rp - \beta\|I_h(\bV)\|^2
\\ \notag
    &\leq -(B(\bV,\bv_1),\bV) - \mu \lp I_h\lp \frac{\bw_1}{\|I_h(\bw_1)\|_{L^2(\Omega)}} - \frac{\bw_2}{\|I_h(\bw_2)\|_{L^2(\Omega)}}\rp,\bw_1 - \bw_2\rp 
\\ \notag
    &=-(B(\bV,\bv_1),\bV) - \mu \lp I_h\lp \frac{\bw_1}{\|I_h(\bw_1)\|_{L^2(\Omega)}} - \frac{\bw_2}{\|I_h(\bw_2)\|_{L^2(\Omega)}}\rp,I_h(\bw_1 - \bw_2)\rp 
\\ \notag
  &=-(B(\bV,\bv_1),\bV) - \mu \left(\mathcal{N}(I_h\bw_1) - \mathcal{N}(I_h\bw_2), I_h(\bw_1) - I_h(\bw_2) \right) 
\\ \notag
\textnormal{where} & \textnormal{ monotonicity now implies we can drop the middle term to obtain}
  \\ \notag
  &\leq -(B(\bV,\bv_1),\bV) 
  \\ \notag
  &\leq c \|\bv_1\|_V\|\bV\|_H\|\bV\|_V 
  \\ \notag
  &\leq \frac{c^2}{2\nu}\|\bv_1\|_V^2\|\bV\|_H^2 + \frac{\nu}{2} \|\bV\|_V^2. 
\end{align}
Hence,
\begin{align*}
 \frac{d}{dt} \|\bV\|_H^2 + \nu \|\bV\|_V^2 &\leq \frac{c^2}{\nu}\|\bv_1\|_V^2\|\bV\|_H^2.
\end{align*}
Integrating in time, we obtain
\begin{align*}
 \|\bV(t)\|_H^2 \leq \|\bV(0)\|_H^2e^{\frac{c^2}{\nu}\int_0^t\|\bv_1\|_V^2 ds} = 0.
\end{align*}
Thus, $\bv_1 = \bv_2$, and strong solutions to \eqref{nonlinDA} are unique.
\end{proof}

\begin{remark}
Notice that if $\beta = 0$, the proof for existence/uniqueness holds for the full range of values of $\|I_h(\bu-\bv)\|_H$.
\end{remark}

\section{Convergence}\label{secConv}

\noindent
In this section, we prove that solutions to \eqref{nonlinDA_Leray} converge to the solution of \eqref{NSEtrue_leray} at least exponentially.  Given a prescribed error $\epsilon >0$, $\bv$ a strong solution to \eqref{nonlinDA_Leray} and $\bu$ a strong solution to \eqref{NSEtrue_leray}, we prove that if $\|\bu-\bv\|$ is not less than epsilon before the exponential convergence of the solutions begins, then there is a small interval in time in which $\|\bv-\bu\|$ converges in finite time at least at a double-exponential rate and in finite time in both the $H$ and $V$ norms up to the chosen small error $\epsilon$. To demonstrate the double-exponential convergence, we use the simple fact that for $y \in (0,1]$, $1-y^{-\gamma} \leq \log(y^\gamma)$.  

For the convergence in $H$, we make the assumption that $I_h$ satisfies \eqref{Ih_always_assumptions},
\eqref{Ih_positive}, and \eqref{Ih_bounded}.
For instance, interpolants given by projection onto low Fourier modes and local averaging over finite volume elements satisfy these conditions.  The proof for the convergence in the $V$ norm holds for the case for interpolants that satisfy \eqref{Ih_always_assumptions} with the additional assumption of \eqref{Ih_A_positive}, which holds, e.g., in the case where $I_h$ is a projection onto low Fourier modes.
Hence, the convergence theorems below
will consider the \eqref{nonlinDA_Leray} initialized with data based on evolving \eqref{linDA_Leray} past a specific, sufficient large time (depending only on known system parameters and observable data).

We now introduce a less restrictive assumption than \eqref{Ih_symmetric} namely, \eqref{Ih_positive}, which will be employed to show the convergence of all strong solutions of \eqref{nonlinDA_Leray} to the corresponding unique reference solution of \eqref{NSEtrue_leray}. Specifically, since we no longer assume that \eqref{Ih_symmetric} holds, we do not necessarily have a unique strong solution to \eqref{nonlinDA_Leray}, but we have global existence by Theorem \ref{thm_existence}. Therefore we will show in Theorem \ref{first_conv_theorem} below that all the strong solutions to \eqref{nonlinDA_Leray} under assumption \eqref{Ih_positive}, regardless of their uniqueness,  converge to the unique strong reference solution of \eqref{NSEtrue_leray}.

\begin{theorem}\label{first_conv_theorem}
 Fix $0< \gamma < 1$.  Let $I_h$ be an interpolant satisfying  \eqref{Ih_always_assumptions}, \eqref{Ih_positive}, and \eqref{Ih_bounded}.  Let $\bv \in C([0,T];V) \cap L^2(0,T;\mathcal{D}(A))$ be a strong solution to \eqref{nonlinDA_Leray} with initial data $\bv_0 \in V$ and time-independent forcing $\bbf \in H$ and $\bu \in C([0,T];V) \cap L^2(0,T;\mathcal{D}(A))$ the strong solution to \eqref{NSEtrue_leray} with initial data $\bu_0 \in V$ and the same forcing $\bbf$. Fix $0 < \epsilon \ll \min\{1,\|\bu(0)-\bv(0)\|_H\}$.  Let $\mu, \beta$ be chosen so that
 \begin{align}\label{first_mu_condition}
  \mu > \max\left\{5c^2\lambda_1G^2\nu, \alpha^\gamma c^2\lambda_1G^2\nu,\frac{\alpha^\gamma}{\gamma}\right\}, \; \beta > c^2\lambda_1G^2\nu
 \end{align}
where $c$ is the specified constant in \eqref{BIN}.  
Let $\epsilon>0$ be given so that
 \begin{itemize}
   \item $\mu c_0 h^2 < \nu/2$,
   \item $\beta c_0 h^2 < \nu$, and
   \item $h \leq \frac{a \alpha^{\gamma}(\epsilon/2)^{\gamma/2}(\nu)^{1-\gamma/2}}{\mu\sqrt{c_0}}$,
 \end{itemize}
where\footnote{It is straight-forward to show that $\gamma\in(0,1)$ implies $a>0$.  A slightly more involved calculation shows that $a\leq \frac12\lambda_1^{(1-\gamma)/2}$.} 
\begin{align*}
a &:= \lp\lp \frac{2-\gamma}{2}\rp^{(2-\gamma)/\gamma} - \lp \frac{2-\gamma}{2}\rp^{2/\gamma}\rp^{\gamma/2}\lambda_1^{(1-\gamma)/2}2^{\gamma/2-1} 
\\
&= (2-\gamma)^{1-\frac{\gamma}{2}}
    \gamma^{\gamma/2}\lambda_1^{(1-\gamma)/2}2^{\frac{\gamma}{2}-2}.
\end{align*}
    Then for all $t \geq \tilde{t}$ (where $\tilde{t}$ is as prescribed in Theorem \ref{ubounds}), $\|\bv-\bu\|_H^2 \to 0$ at least exponentially as in \cite{Azouani_Olson_Titi_2014}.  If $\|\bv(\tilde{t}) - \bu(\tilde{t})\|_H > \epsilon$, then there is a time interval $[t_0, t^*]$ such that $\|\bv-\bu\|_H^2 \to \epsilon$ at a double-exponential rate.  In particular,
\begin{align*}
 \|\bv(t)-\bu(t)\|_H^2 \leq A \exp\pnt{-b\exp\pnt{\mu\alpha^{-\gamma}\gamma(t-t_0)}},
\end{align*}
for all $t < t^*$, where $A:= \exp\pnt{-\frac{1}{\mu\alpha^{-\gamma}\gamma}}$ and $b :=\frac{1}{\mu\alpha^{-\gamma}\gamma} (-(\mu\alpha^{-\gamma}\gamma-1)\log{\|\bw(t_0)\|_H^2}) > 0$.  

\end{theorem}


\begin{proof}
Assume the hypotheses and let $\bw := \bv-\bu$.  
We take the difference of \eqref{NSEtrue_leray} and \eqref{nonlinDA_Leray}, yielding the system
\begin{subequations}
\begin{align}\label{NSE_diff_1}
 \bw_t + B(\bw,\bu) + B(\bv,\bw) + \nu A\bw &= -\mu P_\sigma  \mathcal{N}(\bw) - \beta P_\sigma I_h(\bw) 
 \\
 \bw(\bx,0) &= \bv_0 - \bu_0.
\end{align}
\end{subequations}
We take the action of \eqref{NSE_diff_1} with $\bw$ and use the Lions-Magenes Lemma to obtain
\begin{align}\label{NSE_diff_1_ip}
 \frac{1}{2}\frac{d}{dt}\|\bw\|_H^2 + \ip{B(\bw,\bu)}{\bw} + \nu \|\bw\|_V^2 &= -\mu(\mathcal{N}(\bw),\bw) - \beta(I_h(\bw),\bw)
\end{align}
Suppose without loss of generality that $\|\bw(0)\|_H > 1$.  Since $(I_h\bw, \bw) \geq 0$, the right-hand side of \eqref{NSE_diff_1_ip} is non-positive, and thus
\begin{align}
 \frac{1}{2}\frac{d}{dt}\|\bw\|_H^2 + \ip{B(\bw,\bu)}{\bw} + \nu \|\bw\|_V^2 &\leq - \beta(I_h(\bw),\bw).
\end{align}
Following the analysis of \cite{Carlson_Hudson_Larios_2020}, we obtain the energy estimate
\begin{align*}
    \frac{d}{dt}\|\bw\|_H^2 + \lp \beta - \frac{c^2}{\nu}\|\bu\|_V^2\rp\|\bw\|_H^2 \leq 0.
\end{align*}
Since we have chosen $\beta > c^2 \lambda_1 G^2\nu$, we can continue to follow the analysis of \cite{Carlson_Hudson_Larios_2020} to obtain exponential convergence for $t \geq \tilde{t}$, $\tilde{t}$ being $t_0$ given in Theorem \ref{ubounds}.  

If $\|\bw(\tilde{t})\|_H^2 < \epsilon$, then we are done.  Otherwise, due to the exponential convergence for $t \geq \tilde{t}$ and the fact that $\bw \in C([0,T];V)$, there is an interval $[t_0,t^*]$ over which \[\epsilon < \|\bw(t)\|_H^2 < \min\left\{e^{-\frac{\alpha^\gamma}{\gamma\mu}},\lp\frac{\mu\alpha^{-\gamma}}{\mu\alpha^{-\gamma}+1}\rp^{1/\gamma}\right\}\] for $t\in [t_0,t^*]$.  Thus, denoting $\eta:= \alpha^{-\gamma}$ and utilizing $\|I_h\bw\|_H \leq \alpha\|\bw\|_H$, \eqref{BIN}, the Cauchy-Schwarz inequality, Young's inequality, and Poincar{\'e}'s inequality, for a.e. $t \in [t_0,t^*]$,
\begin{align*}
&\quad
  \frac{1}{2}\frac{d}{dt}\|\bw\|_H^2 + \nu \|\bw\|_V^2 
  \\
  &\leq - \ip{B(\bw,\bu)}{\bw} -\mu \eta(\|\bw\|_H^{-\gamma}I_h(\bw), \bw) - \beta (I_h\bw,\bw) 
  \\
  &\leq - \ip{B(\bw,\bu)}{\bw} -\mu \eta(\|\bw\|_H^{-\gamma}I_h(\bw), \bw) 
  \\
  &\leq -\ip{B(\bw,\bu)}{\bw} + \mu\eta[-\|\bw\|_H^{2-\gamma} + \|\bw\|_H^{-\gamma}(\bw-I_h\bw,\bw)]
  \\
  &\leq c \|\bw\|_H\|\bw\|_V\|\bu\|_V - \mu \eta\|\bw\|_H^{2-\gamma} + \mu\eta\sqrt{c_0}h\|\bw\|_V\|\bw\|_H^{1-\gamma} \\
  &\leq \frac{c^2}{2\nu}\|\bu\|_V^2\|\bw\|_H^2 + \frac{\nu}{2} \|\bw\|_V^2 - \mu\eta \|\bw\|_H^{2-\gamma} + \frac{\mu\eta\sqrt{c_0}h}{\lambda_1^{(1-\gamma)/2}} \|\bw\|_V^{2-\gamma}.
\end{align*}
By Theorem \ref{ubounds}, $\|\bu(t)\|_V^2 \leq 2\lambda_1G^2(\nu)^2$ for all $t \in [t_0, T]$, so condition \eqref{first_mu_condition} implies that $\mu > \frac{c^2}{2\nu}\|\bu\|_V^2$ for all $t \in [t_0, T]$, and hence
\begin{align*}
 \frac{1}{2}\frac{d}{dt}\|\bw\|_H^2 + \frac{\nu}{2}\|\bw\|_V^2 - \frac{\mu\eta\sqrt{c_0}h}{\lambda_1^{(1-\gamma)/2}} \|\bw\|_V^{2-\gamma} &\leq \frac{c^2}{2\nu}\|\bu\|_V^2\|\bw\|_H^2 - \mu \eta\|\bw\|_H^{2-\gamma} \\
 &\leq \mu\eta (\|\bw\|_H^2-\|\bw\|_H^{2-\gamma}).
\end{align*}

We can write expression involving the $\|\bw\|_V$ terms on the left-hand side in the form of $f(x) = ax^2-b(h)x^{2-\gamma}$, where $x$ is taken to be $\|\bw\|_V$.  By Lemma \ref{small_h}, the term $b(h)$ determines the minimum value of $f(x)$ 
and it can be shown via the proof of Lemma \ref{small_h} that $h$ is small enough so that the condition $\frac{\nu}{2}\|\bw\|_V^2 - \frac{\mu\eta\sqrt{c_0}h}{\lambda_1^{(1-\gamma)/2}} \|\bw\|_V^{2-\gamma} \geq -\epsilon/2 $ holds.  Note that $h$ is bounded above by an expression involving the constant $a$.  

As a consequence of our smallness condition on $h$,
\begin{align*}
 \frac{1}{2}\frac{d}{dt} \|\bw\|_H^2 - \epsilon/2 &\leq \mu\eta(\|\bw\|_H^2-\|\bw\|_H^{2-\gamma})
\end{align*}
or simply
\begin{align}\label{bernoulli_ineq_1}
 \frac{d}{dt}\|\bw\|_H^2 \leq 2\mu\eta(\|\bw\|_H^2-\|\bw\|_H^{2-\gamma}) + \epsilon.
\end{align}

Furthermore, we note that the first term on the right-hand side is negative, so applying the fact that for $y\in (0,1]$ we have $1-y^{-\gamma} \leq \log(y^\gamma)$,
\begin{align}\label{super_exp_ineq_1}
 \frac{d}{dt} \|\bw\|_H^2 &\leq 2\mu\eta (1-\|\bw\|_H^{-\gamma}) \|\bw\|_H^2 + \epsilon \notag\\
 &\leq 2\mu\eta(\log{\|\bw\|_H^{\gamma}})\|\bw\|_H^2 + \epsilon \notag \\
 &= \gamma \mu\eta (\log{\|\bw\|_H^2})\|\bw\|_H^2 + \epsilon.
\end{align}

Thus, we have two inequalities \eqref{bernoulli_ineq_1}, a Bernoulli type differential inequality, and \eqref{super_exp_ineq_1}, each of which provides different information.  We analyze \eqref{bernoulli_ineq_1} first to directly obtain convergence to $\epsilon$ in finite time.

By our initial assumptions, we note specifically that $\|\bw(t)\|_H^2 > \epsilon$ for all $t \in[t_0,t^*]$, and therefore for a.e. $t \in [t_0,t^*]$,
\begin{align*}
  \frac{d}{dt} \|\bw\|_H^2 &\leq 2(\mu\eta +1)\|\bw\|_H^2 -2\mu\eta \|\bw\|_H^{2-\gamma}.
\end{align*}
With $z =  \|\bw\|_H^{\gamma}$,
\begin{align*}
\frac{dz}{dt} &\leq \gamma(\mu\eta +1)\left(z-\frac{\mu\eta }{\mu\eta +1}\right),
\end{align*}
which can be rewritten as
\begin{align*}
\frac{d}{dt}\log{\left(\frac{\mu\eta }{\mu\eta +1}-z\right)} &\geq \gamma(\mu\eta +1)
\end{align*}
and integrating from $t_0$ to $t^*$,
\begin{align*}
 z(t^*) &\leq \frac{\mu\eta }{\mu\eta +1} -  \left(\frac{\mu\eta }{\mu\eta +1}-z(t_0)\right) e^{\gamma(\mu\eta +1)(t^*-t_0)},
\end{align*}
or in other words,
\begin{align*}
 \|\bw(t^*)\|_H^{\gamma} &\leq \frac{\mu\eta }{\mu\eta +1} -  \left(\frac{\mu\eta }{\mu\eta +1}-\|\bw(t_0)\|_H^{\gamma}\right) e^{\gamma(\mu\eta +1)(t^*-t_0)}.
\end{align*}
The right-hand side of this inequality approaches $-\infty$ as $t^* \to \infty$.  Note that $t^*$ is fixed, but since we have demonstrated that on this time interval $\|\bw(t)\|_H$ decays in time, we can extend $t^*$ until $\|\bw\|_H^2 = \epsilon$.

We note that the decay rate itself is better characterized by utilizing the inequality \eqref{super_exp_ineq_1}.  
Again, since $\|\bw(t)\|^2 > \epsilon$ for all $t\in [t_0,t^*]$, then for a.e. $t \in [t_0,t^*]$,
\begin{align*}
  \frac{d}{dt}\|\bw\|_H^2 \leq (\gamma \mu\eta (\log{\|\bw\|_H^2})+1)\|\bw\|_H^2.
\end{align*}
Substituting $z = -\log{\|\bw\|_H^2}$, we obtain
\begin{align*}
 \frac{dz}{dt} \geq \gamma \mu\eta z-1,
\end{align*}
which is equivalent to stating that
\begin{align*}
 \frac{d}{dt}\left( \log{(\mu\eta\gamma z-1)}\right) \geq \mu\eta\gamma.
\end{align*}
Integrating over the interval $[t_0,t^*]$, we have that 
\begin{align*}
z(t^*) \geq \frac{1}{\mu\eta\gamma}\left(1 + e^{\log{(\mu\eta\gamma z(t_0)-1)} + \mu\eta\gamma(t^*-t_0)}\right),
\end{align*}
which can be rewritten as
\begin{align*}
 \log{\|\bw(t^*)\|_H^2} &\leq -\frac{1}{\mu\eta\gamma}\left(1 + e^{\log{(-(\mu\eta\gamma-1)\log{\|\bw(t_0)\|_H^2})} + \mu\eta\gamma(t^*-t_0)}\right).
 \end{align*}
 This implies
 \begin{align*}
 \|\bw(t^*)\|_H^2 &\leq A\exp\pnt{-b\exp\pnt{\mu\eta\gamma(t^*-t_0)}},
\end{align*}
where $A := \exp\pnt{-\frac{1}{\mu\eta\gamma}}$ and $b :=\frac{1}{\mu\eta\gamma} (-(\mu\eta\gamma-1)\log{\|\bw(t_0)\|_H^2})$.  
Since this inequality indicates that $\|\bw(t)\|_H^2$ decays monotonically at least double-exponentially, we note that we can extend $t^*$ until $\|\bw(t^*)\|_H^2 = \epsilon$.

Now, we note that convergence to $0$ still holds, since, only using the assumptions on $\mu$ and $h$,
\begin{align*}
  \frac{1}{2}\frac{d}{dt}\|\bw\|_H^2 +\ip{B(\bw,\bu)}{\bw} + \nu\|\bw\|_V^2 
  &= 
  -\mu(P_\sigma\mathcal{N}(I_h\bw),\bw) 
  \\&= 
  -\mu(|I_h\bw|^{-\gamma}I_h\bw,\bw) -\mu(I_h\bw,\bw)
  \\&\leq 
  -\mu (I_h\bw,\bw),
\end{align*}
and therefore by Theorem \ref{CDA_Hnorm}, by our choice of $\mu$ and $h$, convergence to $0$ still holds.
\end{proof}

\begin{remark}\label{remark_eps_barrier}
Note that in Theorem \ref{first_conv_theorem}, convergence in finite time double-exponentially holds by simply analyzing \eqref{super_exp_ineq_1}.  If it was possible for the proof to be improved to shrink $\epsilon$ to $0$, then the inequality \eqref{bernoulli_ineq_1} demonstrates that we would still obtain convergence in the $H$ norm in \textit{finite} time.  The main roadblock keeping us from sending $\epsilon$ to $0$ is that, unlike in the linear-nudging case, where we can employ the inequality $\mu c_0 h^2 < \nu$, in the nonlinear-nudging case, the analogous inequality is $\mu c_0 h^2 \|\bw\|^{-\gamma} < \nu$.  Hence, as $\|\bw\|_H \to 0$, eventually this bound will be violated.
\end{remark}

\begin{corollary}
Assume the hypotheses of Theorem \ref{first_conv_theorem}, except with assumption \eqref{Ih_positive} replaced by assumption \eqref{Ih_symmetric}.  Then the \textit{unique} strong solution of \eqref{nonlinDA_Leray} satisfies the same conclusions of Theorem \ref{first_conv_theorem}.
\end{corollary}
\begin{proof}
    Observe that assumption \eqref{Ih_symmetric} implies \eqref{Ih_positive}. Therefore, by Theorem 3.4 under the more restrictive assumption \eqref{Ih_symmetric} (instead of \eqref{Ih_positive}), \eqref{nonlinDA_Leray} has a global unique strong solution.  By Theorem \ref{first_conv_theorem} this unique strong solution converges to the unique reference strong solution as 
    claimed.
\end{proof}

In the following theorem, we provide a proof of the double-exponential and finite time convergence of $\bv$ to $\bu$ in the $V$ norm.  In this setting, we require a slightly different restriction on the interpolant, namely \eqref{Ih_A_positive}.

\begin{theorem}\label{second_conv_theorem}
 Fix $0<\gamma<1$.  Let $I_h$ satisfy  \eqref{Ih_always_assumptions} and \eqref{Ih_A_positive}. Let $\bv\in C([0,T];V)\cap L^2(0,T;\mathcal{D}(A))$ be a strong solution to \eqref{nonlinDA_Leray} with initial data $\bv_0 \in V$ and forcing $\bbf \in H$ and $\bu \in C([0,T];V)\cap L^2(0,T;\mathcal{D}(A))$ the strong solution to the \eqref{NSEtrue_leray} with initial data $\bu_0 \in V$ and the same forcing $\bbf$.  Fix $0 < \epsilon \ll \min\{1, \|\bu(0)-\bv_0\|_V\}$.  Let $\mu, \beta$ be chosen so
 that 
 \begin{align}\label{second_mu_condition}
 \mu &> \max\left\{(\sqrt{c_0}+\lambda_1^{-1/2})^\gamma c\lambda_1^2(\nu)^2(1+G)^4,\frac{1}{\gamma\lambda^{\gamma/2}}, 3\lambda_1\nu J G\right\},
 \\
 \beta &> 3\lambda_1 \nu JG
 \end{align}
 where $c$ is the constant given by the inequality \eqref{BIN3}, and \[J := 2c \log\left(2c^{3/2}\right)+4c\log(1+G).\]  
 Choose $h$ such that
 \begin{itemize}
  \item $h <1$ (where $1$ has units of length, i.e., it is the linear size of the domain)
  \item $\mu c_0 h^2 < \nu$ and
  \item $h \leq \frac{a (\epsilon/2)^{\gamma/2}(\nu)^{1-\gamma/2}}{\mu\sqrt{c_0}}$,
 \end{itemize}
where $a := \lp\lp \frac{2-\gamma}{2}\rp^{(2-\gamma)/\gamma} - \lp \frac{2-\gamma}{2}\rp^{2/\gamma}\rp^{\gamma/2}\lambda_1^{(1-2\gamma)/2}2^{\gamma/2-1}$. 
Then for all $t > \tilde{t}$ (where $\tilde{t}$ is prescribed in Theorem \ref{ubounds}), $\|\bv-\bu\|_V^2 \to 0$ at least exponentially as in \cite{Azouani_Olson_Titi_2014}.  If $\|\bv(\tilde{t}) - \bu(\tilde{t})\| > \epsilon$, then there is a time interval $[t_0, t^*]$ such that $\|\bv-\bu\|_V^2 \to \epsilon$ at a double-exponential rate.  In particular,
\begin{align*}
 \|\bv(t)-\bu(t)\|_V^2 &\leq Ke^{-be^{\mu\gamma\lambda^{\gamma/2}(t-t_0)}},
\end{align*}
where $K := e^{-\frac{1}{\mu\gamma\lambda^{\gamma/2}}}$ and $b :=\frac{1}{\mu\gamma\lambda^{\gamma/2}} (-(\mu\gamma\lambda^{\gamma/2}-1)\log{\|\bv(t_0)-\bu(t_0)\|_V^2})$.  
\end{theorem}

\begin{remark}
Note that in the case where $I_h = P_m$, the projection onto the Fourier modes of index $m < 1/h$, it is clear that both \eqref{Ih_A_positive} and \eqref{Ih_bounded} hold (with $\alpha=1$), so the hypotheses of the theorem hold in this example.
\end{remark}

\begin{proof}
Let $\bw := \bv-\bu$.  
We take the difference of \eqref{NSEtrue_leray} with \eqref{nonlinDA_Leray}, yielding the system
\begin{align}\label{NSE_diff_2}
 \bw_t + B(\bw,\bu) + B(\bv,\bw) + \nu A\bw &= -\mu P_\sigma\mathcal{N}(I_h(\bw)) - \beta P_\sigma I_h(\bw) \\
 \bw(\bx,0) &= \bv_0 - \bu_0\notag.
\end{align}
Taking the $H$ inner-product of \eqref{NSE_diff_2} with $A\bw$ and applying the Lions-Magenes Lemma, we obtain
\begin{align*}
&\quad
 \frac{1}{2}\frac{d}{dt} \|\bw\|_V^2 + (B(\bu,\bw),A\bw) + (B(\bw,\bu),A\bw) + (B(\bw,\bw),A\bw) + \nu \|A\bw\|_H^2 
 \\&= 
 -\mu (\mathcal{N}(\bw),A\bw) - \beta (I_h\bw, A\bw),
\end{align*}
which can be rewritten (using \eqref{jacobi}) as 
\begin{align*}
 \frac{1}{2}\frac{d}{dt} \|\bw\|_V^2 - (B(\bw,\bw),A\bu) + \nu \|A\bw\|_H^2 =-\mu (\mathcal{N}(\bw),A\bw) - \beta (I_h\bw, A\bw)
\end{align*}

Suppose without loss of generality $\|\bw(0)\|_V > 1$.  Via assumption \eqref{Ih_A_positive},
\begin{align*}
 \frac{1}{2}\frac{d}{dt} \|\bw\|_V^2 - (B(\bw,\bw),A\bu) + \nu \|A\bw\|_H^2 \leq - \beta (I_h\bw, A\bw).
\end{align*}
Following the analysis of \cite{Azouani_Olson_Titi_2014}, we obtain the estimate
\begin{align}
    \frac{d}{dt}\|\bw\|_V^2 + \frac12\left[\beta - \frac{J^2}{\beta}\|A\bu\|_H^2\right] \|\bw\|_V^2 \leq 0.
\end{align}
Since $\beta > 3\lambda_1 \nu JG$, we can continue to follow the analysis to obtain exponential convergence for $t \geq \tilde{t}$, where $\tilde{t}$ is $t_0$ from Theorem \ref{ubounds}.
If $\|\bw(\tilde{t})\|_V^2 < \epsilon$, then we are done.   Otherwise, due to the exponential convergence for $t \geq \tilde{t}$ and the fact that $\bw \in C([0,T];V)$, there is an interval $[t_0,t^*]$ over which $\epsilon < \|\bw(t)\|_H^2 <\min\left\{e^{-\frac{1}{\mu\gamma\lambda_1^{\gamma/2}}},\lp\frac{\mu\lambda_1^{\gamma/2}}{\mu\lambda_1^{\gamma/2}+1}\rp^{1/\gamma}\right\}$ for $t\in [t_0,t^*]$.  Using $\|A\bu\|_H^2 \leq  c\lambda_1^2(\nu)^2(1+G)^4$ and $(I_h\bw,A\bw) \geq 0$, we have that for a.e. $t \in [t_0,t^*]$,
\begin{align*}
&\quad
 \frac{1}{2} \frac{d}{dt} \|\bw\|_V^2 + \nu \|A\bw\|_H^2 
 \\
 &\leq -(B(\bw,\bw),A\bu) - \mu c(h) \|\bw\|_V^{-\gamma}(I_h\bw,A\bw) - \beta(I_h\bw,A\bw)
 \\
 &\leq -(B(\bw,\bw),A\bu) - \mu c(h) \|\bw\|_V^{-\gamma}(I_h\bw,A\bw)
 \\
 &= -(B(\bw,\bw),A\bu) +\mu c(h) [ (1-\|\bw\|_V^{-\gamma})\|\bw\|_V^2 - \|\bw\|_V^2 \\
 &\quad + \|\bw\|_V^{-\gamma}(\bw-I_h\bw,A\bw)] \\
 &\leq c\|\bw\|_H^{1/2}\|A\bw\|_H^{1/2}\|\bw\|_V\|A\bu\|_H  \\
 &\phantom{=} + \mu c(h)(1-\|\bw\|_V^{-\gamma})\|\bw\|_V^2 \\
 &\phantom{=} + \mu c(h) \sqrt{c_0}h \|\bw\|_V^{-\gamma}\|\bw\|_V\|A\bw\|_H - \mu c(h) \|\bw\|_V^2 \\
  &\leq c\lambda_1\|A\bw\|_H\|\bw\|_V\|A\bu\|_H  \\
 &\phantom{=} + \mu c(h)(1-\|\bw\|_V^{-\gamma})\|\bw\|_V^2 \\
 &\phantom{=} + \mu c(h) \sqrt{c_0}h \|\bw\|_V^{1-\gamma}\|A\bw\|_H - \mu c(h) \|\bw\|_V^2 \\
&\leq \frac{c\lambda_1^2}{2\nu}\|\bw\|_V^2\|A\bu\|_H^2 + \frac{\nu}{2} \|A\bw\|_H^2 \\
 &\phantom{=} + \mu c(h)(1-\|\bw\|_V^{-\gamma})\|\bw\|_V^2 + \mu c(h) \sqrt{c_0}h \|\bw\|_V^{1-\gamma}\|A\bw\|_H \\
 &\phantom{=} - \mu c(h) \|\bw\|_V^2,
\end{align*}
where $c(h) = (\sqrt{c_0}h+\lambda_1^{-1/2})^{-\gamma}$.  Since $h <1$, it follows that $c(h) \geq(\sqrt{c_0}+\lambda_1^{-1/2})^{-\gamma} := \eta$. Note also that this constant can be bounded above by a constant independent of $h$, specifically, $c(h) \leq (\lambda^{-1/2})^{-\gamma} = \lambda_1^{\gamma/2}$.  Hence
\begin{align*}
&\quad
 \frac{1}{2}\frac{d}{dt} \|\bw\|_V^2 + \frac{\nu}{2} \|A\bw\|_H^2 + \left[\mu \eta - \frac{c\lambda_1^2}{2\nu}\|A\bu\|_H^2\right]\|\bw\|_V^2
 \\&\leq 
 \mu \eta (1-\|\bw\|_V^{-\gamma})\|\bw\|_V^2  + \frac{\mu  \sqrt{c_0}h}{\lambda_1^{(1-2\gamma)/2}} \|A\bw\|_H^{2-\gamma}.
\end{align*}
Since $\|A\bu(t)\|_H^2 \leq c \lambda_1^2(\nu)^2(1+G)^4$ for all $t \in[t_0, T]$ due to Theorem \ref{ubounds}, then the condition \eqref{second_mu_condition} implies that $\mu \eta - \frac{c\lambda_1^2}{2\nu}\|A\bu\|_H^2 \geq 0$.  Secondly, the same reasoning in Theorem \ref{first_conv_theorem} utilizing Lemma \ref{small_h} shows that $h$ is sufficiently small so that $\left(\frac{\nu}{2}\|A\bw\|_H^2 - \frac{\mu  \sqrt{c_0}h}{\lambda_1^{(1-2\gamma)/2}} \|A\bw\|_H^{2-\gamma}\right) \geq -\epsilon/2$ for our given tolerance $\epsilon > 0$.  
Hence, we obtain the inequality
\begin{align}\label{bernoulli_ineq}
 \frac{d}{dt} \|\bw\|_V^2 &\leq 2\mu \lambda_1^{\gamma/2} (1-\|\bw\|_V^{-\gamma})\|\bw\|_V^2 + \epsilon.
\end{align}

Furthermore, we note that the first term on the right-hand side is negative, and applying the fact that for $y\in (0,1]$ we have $1-y^{-\gamma} \leq \log(y^\gamma)$, we note that 
\begin{align}\label{super_exp_ineq}
 \frac{d}{dt} \|\bw\|_V^2 &\leq 2\mu\lambda^{\gamma/2}(\log{\|\bw\|_V^{\gamma}})\|\bw\|_V^2 + \epsilon \notag\\
 &= \gamma\mu\lambda^{\gamma/2}(\log{\|\bw\|_V^2})\|\bw\|_V^2 + \epsilon.
\end{align}
Thus, we have two inequalities analogous to those in Theorem \ref{first_conv_theorem}.  We once again analyze \eqref{bernoulli_ineq} first to directly obtain convergence to $\epsilon$ in finite time. 
By our initial assumptions, we note specifically that $\|\bw(t)\|_V^2 > \epsilon$ for all $t\in [t_0,t^*]$, and therefore for a.e. $t\in [t_0,t^*]$,

\begin{align*}
  \frac{d}{dt} \|\bw\|_V^2 &\leq 2(\mu \lambda_1^{\gamma/2}+1)\|\bw\|_V^2 -2\mu \lambda_1^{\gamma/2}\|\bw\|_V^{2-\gamma}.
\end{align*}
Using the same methods as in Theorem \ref{first_conv_theorem}, we obtain
\begin{align*}
 \|\bw(t^*)\|_V^{\gamma} &\leq \frac{\mu \lambda_1^{\gamma/2}}{\mu \lambda_1^{\gamma/2}+1} -  \left(\frac{\mu\lambda_1^{\gamma/2}}{\mu \lambda_1^{\gamma/2}+1}-\|\bw(t_0)\|_V^{\gamma}\right) e^{\gamma(\mu \lambda_1^{\gamma/2}+1)(t^*-t_0)},
\end{align*}
 and again note the right-hand side of this inequality approaches $-\infty$ as $t^* \to \infty$.  Note that $t^*$ was chosen fixed, but since we have demonstrated that on this time interval that $\|\bw(t)\|_V$ decays in time, we can extend $t^*$ until $\|\bw\|_V^2 = \epsilon$.

 As in Theorem \ref{first_conv_theorem} we note that the decay rate itself is better characterized by utilizing the inequality \eqref{super_exp_ineq}.  Since $\|\bw\|_V^2 > \epsilon$ for all $t \in [t_0,t^*]$, then for a.e. $t \in [t_0,t^*]$,
\begin{align*}
  \frac{d}{dt}\|\bw\|_V^2 \leq (\gamma \mu\lambda^{\gamma/2}(\log{\|\bw\|_V^2})+1)\|\bw\|_V^2.
\end{align*}
Following similar steps to those in the proof of Theorem \ref{first_conv_theorem}, we arrive at
\begin{align*}
 \log{\|\bw(t^*)\|_V^2} &\leq -\frac{1}{\mu\eta\gamma\lambda^{\gamma/2}}\left(1 + e^{\log{(-(\mu\gamma\lambda^{\gamma/2}-1)\log{\|\bw(t_0)\|_V^2})} + \mu\gamma\lambda^{\gamma/2}(t^*-t_0)}\right) \\
 \|\bw(t^*)\|_V^2 &\leq Ke^{-be^{\mu\gamma\lambda^{\gamma/2}(t^*-t_0)}},
\end{align*}
where $K := e^{-\frac{1}{\mu\gamma\lambda^{\gamma/2}}}$ and $b :=\frac{1}{\mu\gamma\lambda^{\gamma/2}} (-(\mu\gamma\lambda^{\gamma/2}-1)\log{\|\bw(t_0)\|_V^2})$.  
Since this inequality indicates that $\|\bw\|_V^2$ decays monotonically at least double-exponentially, we again note that we can extend $t^*$ until $\|\bw(t^*)\|_V^2 = \epsilon$.

In addition, note that with these assumptions on the interpolant we can directly obtain double-exponential and finite-in-time $L^2$ convergence of $\|\bw\|_H^2$ to $\epsilon/\lambda_1$ due to Poincar{\'e}'s inequality.

We again note that 
\begin{align*}
&\quad
 \frac{1}{2}\frac{d}{dt}\|\bw\|_V^2 - ( B(\bw,\bw),A\bu) + \nu \|A\bw\|_H^2 
 \\&= 
 -\mu\left(\mathcal{N}(I_h(\bw)),A\bw\right)
 \\&= 
 -\mu\left(|I_h\bw|^{-\gamma}I_h(\bw),A\bw\right) -\mu\left(I_h(\bw),A\bw\right),
 \\&\leq 
 -\mu\left(I_h(\bw),A\bw\right),
\end{align*}
using the assumption that $\left(I_h(\bw),A\bw\right) \geq 0$.
By our choice of $\mu$ and $h$, we have by Theorem \ref{CDA_Vnorm} that exponential convergence still holds.
\end{proof}

\begin{remark}\label{rmk_linear_off}
Instead of considering the nonlinear-nudging CDA algorithm implemented for all time, one could alternatively consider the case where fewer data points are observed initially and utilize the linear-nudging CDA algorithm up until a computable time $T$ (see Appendix \ref{sec_comp}) where $\|\bw\| < 1$ for either the $H$ or $V$ norm (where the exact upper bound is what is given in the hypotheses of Theorems \ref{first_conv_theorem} and \ref{second_conv_theorem} above).  This nonlinear term would then be given by setting $\beta = 0$ and $\mathcal{N}$ redefined as
\begin{align*}
\mathcal{N}(\bphi) := 
\begin{cases}
0, &  \text{ if }\quad\|\bphi\|_{L^2(\Omega)} = 0,
\\ 
\bphi\|\bphi\|_{L^2(\Omega)}^{-\gamma}, & \text{ if }\quad 0<\|\bphi\|_{L^2(\Omega)} <1,
\\
\bphi, & \text{ if }\quad 1\leq\|\bphi\|_{L^2(\Omega)}.
\end{cases}
\end{align*}
Then, one could ``turn on'' the nonlinearity by initializing the nonlinear-nudging CDA system with data from the linear-nudging CDA system.  In this setting, the $h$ for the linear-nudging data assimilation is fixed, and then, depending on the choice of $\epsilon$, one can determine whether to maintain or decrease $h$ (or refine the grid on which one is interpolating) in order to always guarantee double-exponential convergence.  In other words, the error of the convergence prescribed requires a tuning of the accuracy of the interpolant: the smaller the error, the smaller we required $h$ to be, i.e. the more accurate the interpolant needed to be.  For example, in the case of Fourier truncation, one would need a greater number of observed wave modes, and in the case of volume interpolation, one would have to have knowledge of the average of the solution over smaller volumes covering the domain.  This implementation of the linear-nudging CDA algorithm and subsequently the nonlinear-nudging CDA algorithm could be implemented computationally as well, where the time to switch between the linear-nudging CDA algorithm and the nonlinear-nudging CDA algorithm (with or without the linear piece) is computed in Appendix \ref{sec_comp} below.
\end{remark}

\begin{remark}\label{type2_interp_rmk}
    One could also work through similar existence and convergence arguments for type 2 interpolants, where $I_h$ instead satisfies the bound
    \begin{align}\label{interp_bound_type2}
        \|\bphi - I_h(\bphi)\|_{L^2(\Omega)}^2 
        \leq 
        \tfrac14 c_0^2h^4\|\bphi\|_{H^2(\Omega)}^2.
    \end{align}
    However, it is not very illuminating nor does it necessarily expand our possible choice of interpolants, as the methods of proof for the super-exponential convergence rely most heavily on the other assumptions being made on $I_h$, notably, in Theorem \ref{first_conv_theorem} the proof of the super-exponential convergence relies exclusively on the bounds \eqref{Ih_positive}, \eqref{Ih_bounded}, while in Theorem \ref{second_conv_theorem} the proof for the super-exponential convergence relies exclusively on the condition \eqref{Ih_A_positive} and \eqref{Ih_always_assumptions}.  In particular, one needs that the nonlinear weight can be bounded in the $H$ and $V$ norms, respectively, which is not provided by the bound \eqref{interp_bound_2}.
\end{remark}

\section{Computational Results}\label{sec_simulations}
In this section, we present some simulations of the nonlinear-nudging data assimilation algorithm discussed above, in the context of the 2D incompressible Navier-Stokes equations with periodic boundary conditions, and forcing over a wide range of scales.  In particular, we demonstrate that the convergence rate is super-exponential in time, until the error becomes quite small ($\|\psi-\psi_{\text{DA}}\|_{L^2} \approx 5\times10^{-12}$ in our trials, see notation below), at which point the convergence becomes merely exponential, as discussed in Remark \ref{remark_eps_barrier} and Appendix \ref{sec_Heuristic_eps_argument}.
The results are shown in \eqref{fig_super_exp_conv}.

All simulations were carried out using pseudo-spectral methods at the stream-function level in our own Matlab code, and run using Matlab version 2020b.  The mean-free stream functions $\psi$ and $\psi_{\text{DA}}$ were determined by $\nabla^\perp \psi=\bu$ and $\nabla^\perp \psi_{\text{DA}}=\bv$. Fourier transforms were computed using Matlab's \texttt{fftn} tool.  The linear viscosity term was handled implicitly using an integrating factor method Euler algorithm, as described in, e.g.,  \cite{Kassam_Trefethen_2005}.  For the interpolation operator $I_h$, we used a projection onto low Fourier modes.  We used a uniform time step of $\Delta t = 3.1250\times 10^{-4}$, which is sufficient to satisfy the advective CFL constraint.  The nonlinear term was treated explicitly (respecting the 2/3's dealiasing rule), using the Basdevant formulation (see, e.g., \cite{Basdevant_1983_JCP,Emami_Bowman_2018}).  The periodic domain was $[-\pi,\pi)^2$ with a uniform mesh of $1024^2$ grid points.   Initial data $\bu_0$ for the ``true'' simulation generated by starting with zero initial data, and then running the simulation until the energy, and enstrophy, appeared to be in an approximately statistically steady state (judged visually), which happened at $t \approx 240$. The energy spectrum of the initial data $\psi_0$, and the corresponding vorticity ($\triangle\psi_0$) are pictured in Figure \ref{fig_u0_and_f}.

\begin{figure}
\centering
\begin{subfigure}[t]{.48\textwidth}
  \centering
  \includegraphics[width=\textwidth,trim = 15mm 5mm 15mm 8mm, clip]{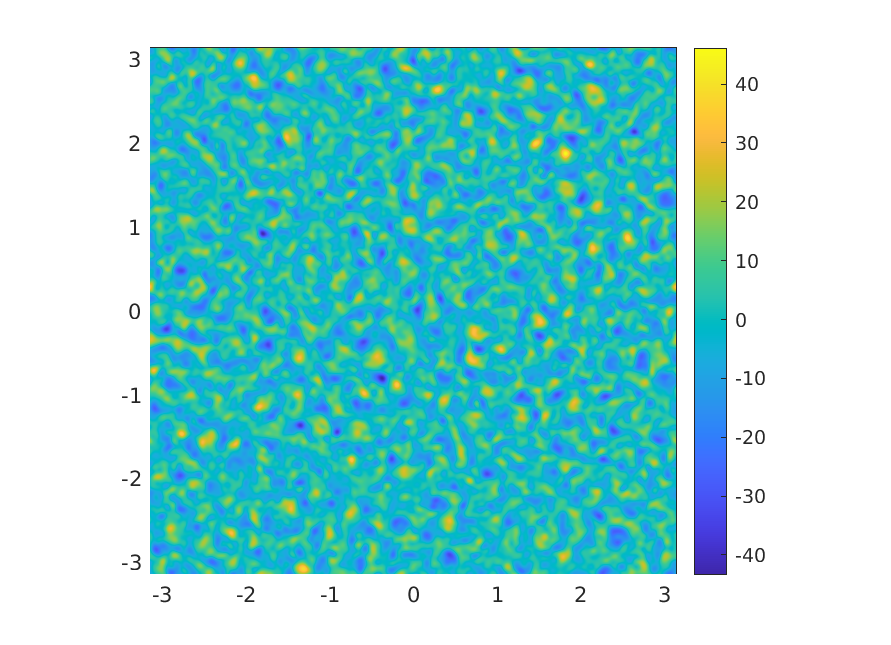}
  \caption{Curl of initial velocity}
  \label{fig_init_vor}
\end{subfigure}
\begin{subfigure}[t]{.48\textwidth}
  \centering
  \includegraphics[width=\textwidth,trim = 0mm 0mm 0mm 0mm, clip]{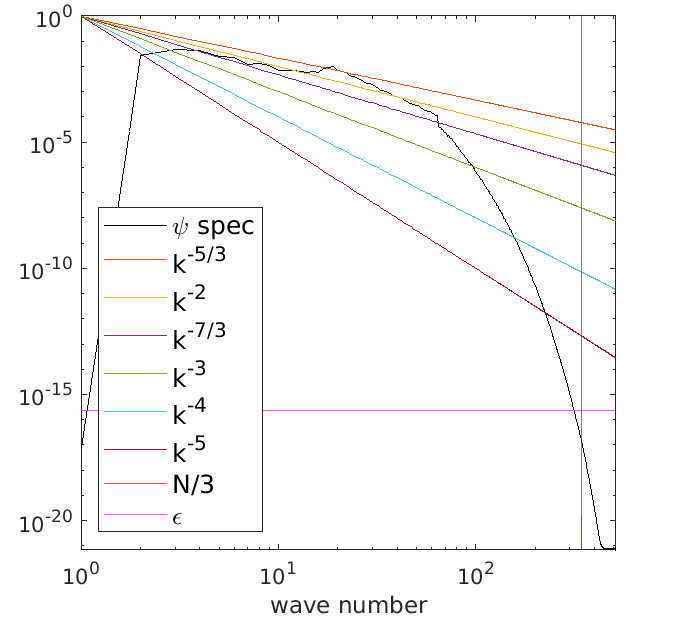}
  \caption{Spectrum of initial $\psi$}
  \label{fig_init_spec}
\end{subfigure}
\begin{subfigure}[t]{.48\textwidth}
  \centering
  \includegraphics[width=\textwidth,trim = 15mm 5mm 15mm 8mm, clip]{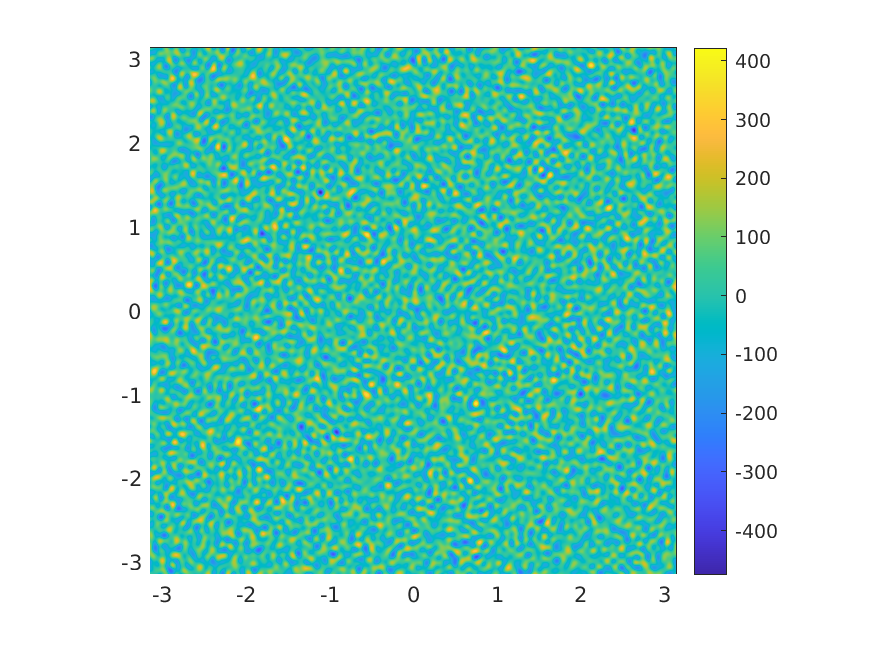}
  \caption{Curl of the forcing}
  \label{fig_f_curl}
\end{subfigure}
\begin{subfigure}[t]{.48\textwidth}
  \centering
  \includegraphics[width=\textwidth,trim = 0mm 0mm 0mm 0mm, clip]{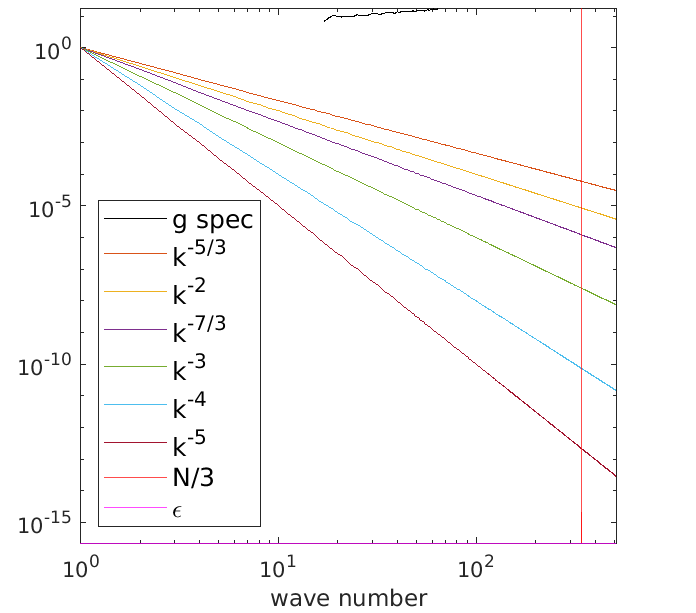}
  \caption{Spectrum of curl of forcing}
  \label{fig_f_curl_spec}
\end{subfigure}
	\caption{\label{fig_u0_and_f} 
	Initial data and forcing.}
\end{figure}


As for the forcing, in light of Remark \ref{remark_eps_barrier}, we were interested in a time-independent force which injects energy at high wave modes in order to better see the effect of the nonlinear-nudging data assimilation term (see \ref{sec_Heuristic_eps_argument} for further rationale).  Therefore, we determined a forcing by choosing normally distributed random values for the real and complex part of each Fourier coefficient $\bbf_\bk$ of the force with wavenumber between $16$ and $64$; namely, the set $\{\bk=(k_1,k_2)|\, 16^2\leq k_1^2$ $+k_2^2\leq64^2\}$ (in fact, only half of the wavemodes were assigned and the rest were computed using the reality condition $\bbf_{-\bk}=\overline{\bbf_{\bk}}$).  Matlab's random number generator was initialized using \texttt{rng(0)} for consistency and reproducibility.  The curl of the forcing, and its energy spectrum, are pictured in Figure \ref{fig_u0_and_f}.

Our parameters were chosen as follows: The Grashof number was $G=250,000$, the viscosity was $\oldnu=0.008$, and $h$ was chosen so that wavemodes of wavenumbers less than or equal to $32$ were observed; that is, the wavemodes at wavenumbers $\set{\bk=(k_1,k_2)|\, 0<k_1^2+k_2^2\leq32^2}$ were observed.  For the nonlinear-nudging data assimilation parameters, we choose $\gamma = 0.1$, and $\mu = \beta = 2$.  These parameter ranges were not finely tuned to exhibit any special behavior, other than avoiding instability (seen, e.g., when $\mu$ is too large).  In our own tests (data not reported here), we observed that modest changes in these parameter values did not yield significant qualitative changes in the results, indicating qualitative robustness of the results.

\begin{figure}
\centering
\begin{subfigure}[t]{1\textwidth}
  \centering
  \includegraphics[width=\textwidth,trim = 0mm 0mm 0mm 0mm, clip]{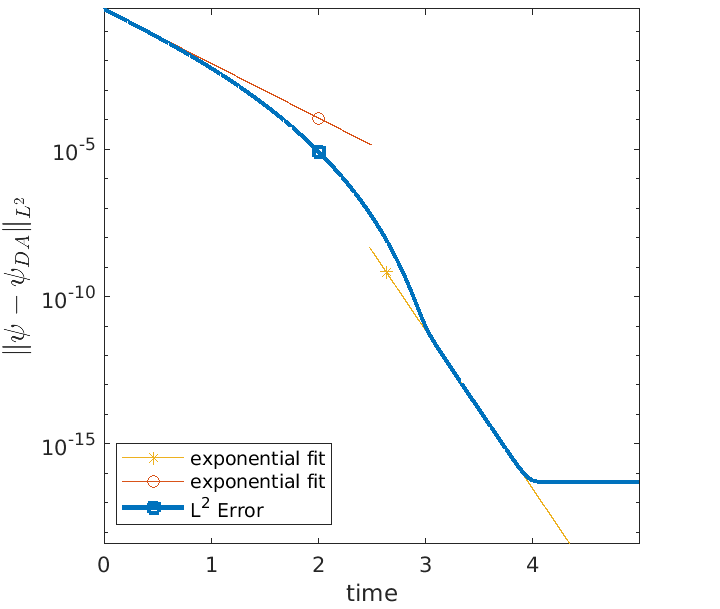}
  \caption{(log-linear plot) Blue curve: $L^2$ error of the solution.  Red line: Exponential fit up to time $t\approx0.15$.  Orange line:  Exponential fit for $3.1\lesssim t\lesssim 3.8$.}
  \label{fig_super_exp_conv}
\end{subfigure}
\begin{subfigure}[t]{.48\textwidth}
  \centering
  \includegraphics[width=\textwidth,trim = 0mm 0mm 0mm 0mm, clip]{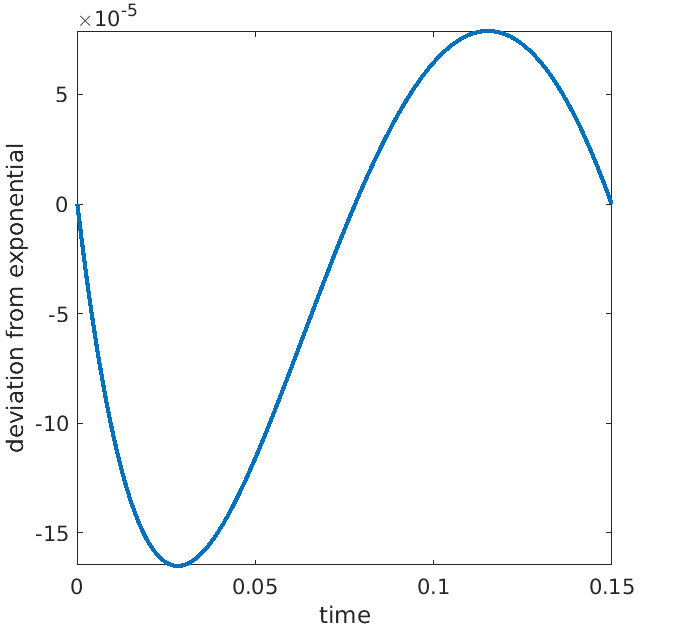}
  \caption{(linear-linear plot) Error minus (red line) exponential fit.}
  \label{fig_super_exp_conv_upper}
\end{subfigure}
\begin{subfigure}[t]{.48\textwidth}
  \centering
  \includegraphics[width=\textwidth,trim = 0mm 0mm 0mm 0mm, clip]{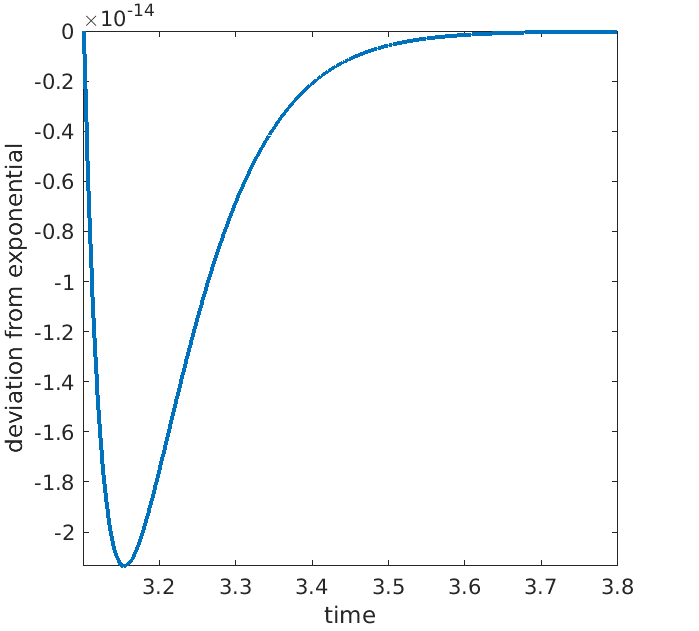}
  \caption{(linear-linear plot) Error minus (orange line) exponential fit.}
  \label{fig_super_exp_conv_lower}
\end{subfigure}
	\caption{\label{fig_conv} 
	(A): $L^2$ error between the assimilated and true solutions.    (B) and (C): deviations from exponential fits.}
\end{figure}

\FloatBarrier

These graphics in particular corroborate our analysis in terms of the failing of super-exponential convergence due the $\epsilon$-barrier being reached, as seen in Figure \ref{fig_conv}.  In particular, in Figure \ref{fig_super_exp_conv}, convergence appears exponential at early times, then becomes super-exponential, and finally returns to merely exponential at later times.  The deviations from an exponential fit was observed to be fairly small: $1.7\times 10^{-4}$ at early times (see Figure \ref{fig_super_exp_conv_upper}) and $2.2\times 10^{-14}$ at later times.  In between these times, super-exponential convergence is observed (see Figure \ref{fig_super_exp_conv}).


\section{Conclusion}\label{sec_conclusion}
\noindent
In this paper, we proved the existence of solutions to the nonlinear-nudging data assimilation system under the same assumptions on the interpolant as that of the linear-nudging data assimilation system.  Uniqueness of solutions were proven to hold under more stringent assumptions on the interpolant operator of the observed measurements.  With different assumptions on the interpolant, convergence of any corresponding solution to the nonlinear-nudging data assimilation to the true solution of the 2D incompressible Navier-Stokes equations was shown to be obtained up to a prescribed error in finite time at an at least double-exponential rate.  In particular, any solution of the nonlinear-nudging system, even in regimes where uniqueness might not hold, will converge to the true solution.
These results provide a theoretical foundation for the computational results seen in simulations in \cite{Larios_Pei_2017_KSE_DA_NL, Hudson_Jolly_2019}, and the present work.

\section{Appendix}

\subsection{Heuristic argument for the $\epsilon$-barrier}\label{sec_Heuristic_eps_argument}
We analyze  \eqref{generic_DE_nonlin_CDA_sum_CDA} in the Navier-Stokes case \eqref{NSEtrue_leray}, i.e.,  $\mathcal{F}(\bv) = -B(\bv,\bv) - \nu A\bv + \bbf$, 
with $I_h = P_m$, i.e., projection onto the low Fourier modes of index $m < 1/h$.  This yields the equation
\begin{align}\label{generic_DE_nlCDA_specificIH}
    \bv_t + \nu A \bv = B(\bv,\bv) - \mu \|P_m(\bu-\bv)\|_{L^2(\Omega)}^{-\gamma}P_m(\bv-\bu) - \beta P_m(\bv-\bu)
\end{align}
Set $\bw = \bv - \bu$.  Subtracting \eqref{generic_DE_nlCDA_specificIH} from the reference system, one obtains
\begin{align}
    \bw_t + \nu A \bw = B(\bv,\bv)-B(\bu,\bu) -\mu\|P_m\bw\|_{L^2(\Omega)}^{-\gamma}P_m\bw - \beta P_m\bw.
\end{align}
Taking a (formal) inner-product with $\bw$ and simplifying yields
\begin{align}\label{formal_ip_generic}
    &\quad
    \frac12 \frac{d}{dt} \|\bw\|_{L^2(\Omega)}^2 + \nu\|A^{1/2}\bw\|_{L^2(\Omega)}^2 
    \\&= \notag
    \ip{B(\bv,\bv)-B(\bu,\bu)}{\bw} - \mu \|P_m\bw\|^{-\gamma}\|P_m\bw\|_{L^2(\Omega)}^2 - \beta \|P_m\bw\|_{L^2(\Omega)}^2.
\end{align}
Denoting $Q_m:=I-P_m$ and noting that $\|\bw\|_{L^2}^2 = \|P_m\bw\|_{L^2}^2+\|Q_m\bw\|_{L^2}^2$, 
\begin{align}
&\quad \notag
    \frac12 \frac{d}{dt} \|\bw\|_{L^2(\Omega)}^2 + \nu\|A^{1/2}\bw\|_{L^2(\Omega)}^2 + \beta\|\bw\|_{L^2(\Omega)}^2 + \mu\|P_m\bw\|_{L^2(\Omega)}^{-\gamma}\|\bw\|_{L^2(\Omega)}^2 
\\& = \notag
    \ip{B(\bv,\bv)-B(\bu,\bu)}{\bw} + \mu \|P_m\bw\|_{L^2(\Omega)}^{-\gamma}\|Q_m\bw\|_{L^2(\Omega)}^2 + \beta \|Q_m\bw\|_{L^2(\Omega)}^2
\end{align}
Rearranging, we obtain
\begin{align}\label{destabilization_calculation}
&\quad
    \frac12 \frac{d}{dt} \|\bw\|_{L^2(\Omega)}^2 + \frac{\nu}{2}\|P_m A^{1/2}\bw\|_{L^2(\Omega)}^2 + \frac{\nu}{2}\|A^{1/2}\bw\|_{L^2(\Omega)}^2 
    \\&\quad \notag
    +  \beta\|\bw\|_{L^2(\Omega)}^2 + \mu\|P_m\bw\|_{L^2(\Omega)}^{-\gamma}\|\bw\|_{L^2(\Omega)}^2 
\\& = \notag
    \ip{B(\bv,\bv)-B(\bu,\bu)}{\bw} + \mu \|P_m\bw\|_{L^2(\Omega)}^{-\gamma}\|Q_m\bw\|_{L^2(\Omega)}^2 
\\&\quad \notag
    + \beta \|Q_m\bw\|_{L^2(\Omega)}^2 - \frac{\nu}{2}\|Q_mA^{1/2}\bw\|_{L^2(\Omega)}^2
\end{align}
We leave part of the dissipation on the left-hand side to absorb terms bounding $B(\bv,\bv) - B(\bu,\bu)$.  The terms $\|Q_m\bw\|_{L^2(\Omega)}^{2}$ and $\|Q_m\bw\|_{L^2(\Omega)}^{2-\gamma}$ are the terms hindering the exponential convergence of $\bv$ to the reference solution $\bu$, and hence we want these last three terms to be negative overall.  We expand the last three terms on the right-hand side to obtain
\begin{align}
    \sum\limits_{|\bk| = m+1}^\infty (\mu\|P_m\bw\|_{L^2}^{-\gamma} + \beta - \frac{\nu}{2}|\bk|^2) |\hat{\bw}_\bk|^2 \leq 0.
\end{align}
No matter how large one takes $m$ (i.e., how small $h$ is taken, since we generally take $m \sim L/h^2$, where $L$ is a characteristic length scale), as $\|\bw\|_{L^2(\Omega)} \to 0$, $\|P_m\bw\|_{L^2}^{-\gamma} \to \infty$, indicating there is a time at which the error becomes small enough that this term will hurt the rate of convergence rather than help.  Moreover, we see from \eqref{destabilization_calculation} that the larger $\beta$ is chosen (e.g., in order to enhance the convergence rate of the small scales), the more strongly the small scales (as measured by $\|Q_m\bw\|_{L^2(\Omega)}^2$) are destabilized.  This appears to be the reason why the super-exponential convergence rate is eventually destroyed, as seen both in our analysis and in our simulations.  We refer to this as a ``spill-over'' effect; namely, the phenomenon that increased control of the large scales leads to increased destabilization of the small scales.  In the case of the Navier-Stokes equations, the spill-over of energy into the small scales is controlled by the presence of viscosity; namely, for large enough $\nu$, the error in the small scales is damped strongly enough to counteract the spill-over effect.  

Marvelously, in the Navier-Stokes case, the exponential convergence still holds in spite of the spill-over effect.  This can be seen by writing \eqref{formal_ip_generic} as
\begin{align}\label{destabilization}
     &\quad
    \frac12 \frac{d}{dt} \|\bw\|_{L^2(\Omega)}^2 + \nu\|A^{1/2}\bw\|_{L^2(\Omega)}^2
    \\&= \notag
    \ip{B(\bv,\bv)-B(\bu,\bu)}{\bw}  - (\beta+\mu) \|P_m\bw\|_{L^2(\Omega)}^2 + \mu (\|P_m\bw\|_{L^2(\Omega)}^2  - \|P_m\bw\|_{L^2(\Omega)}^{2-\gamma});
\end{align}
the final term becomes negative as $\|\bw\|_{L^2(\Omega)}\to 0$, and hence exponential convergence is maintained with an improved rate than for the standard linear-nudging CDA algorithm thanks to the added $\mu$ in the linear term. In other words, although the rate of convergence is no longer super exponential, the nonlinear-nudging term does not become so malicious as $\|\bw\|_{L^2(\Omega)} \to 0$ that it counteracts the standard exponential convergence and in fact it still improves the exponential rate of convergence.  This further elucidates our comments in Remark \ref{rmk_linear_off}, i.e., an exponential rate of convergence can be maintained by the nonlinear-nudging term alone.

\subsection{Proof of Lemma \ref{small_h}}\label{sec_small_h}

\begin{proof}
 Let $\delta := \min\left\{ a/2, a^{\frac{2-\gamma}{2}}\lp \frac{\epsilon}{\lp\frac{2-\gamma}{2}\rp^{\frac{2-\gamma}{\gamma}} - \lp\frac{2-\gamma}{2}\rp^{\frac{2}{\gamma}}} \rp^{\frac{\gamma}{2}}\right\}$.  Note that $f(x)$ has two critical points at $x = 0$ and $x = \lp\frac{(2-\gamma)\delta}{2a}\rp^{\frac{1}{\gamma}}$. 
 We further note that $f(x_0)$ is a global minimum since $f(0) = 0$, 
 \begin{align*}
  f(x_0) &= a \lp\lp\frac{(2-\gamma)\delta}{2a}\rp^{\frac{1}{\gamma}}\rp^2 - \delta \lp \lp\frac{(2-\gamma)\delta}{2a}\rp^{\frac{1}{\gamma}} \rp^{2-\gamma} \\
  &= a^{-\frac{2-\gamma}{\gamma}}\delta^{\frac{2}{\gamma}} \lp \lp\frac{2-\gamma}{2}\rp^{\frac{2}{\gamma}} - \lp\frac{2-\gamma}{2}\rp^{\frac{2-\gamma}{\gamma}}\rp 
  \leq 0,
 \end{align*}
$f'(x) \geq 0 $ for all $x \geq x_0$, and $f'(x) \leq 0$ for $x \leq x_0$. Indeed, $f'(x) \geq 0$ for all $x \geq x_0$ since
\begin{align*}
 f'(x) &\geq 2ax - (2-\gamma)\delta x_0^{1-\gamma} \\
 &= 2ax - (2-\gamma)\delta \lp\frac{(2-\gamma)\delta}{2a}\rp^{\frac{1-\gamma}{\gamma}} \\
 &= 2a\lp x - \lp\frac{(2-\gamma)\delta}{2a}\rp^{\frac{1}{\gamma}}\rp \\
 &= 2a(x-x_0)\\
 &\geq 0,
\end{align*}
and $f'(x) \leq 0$ for all $x \leq x_0$ since 
\begin{align*}
 f'(x) &\leq 2ax_0 - (2-\gamma) \delta x^{1-\gamma}\\
 &= 2a\lp\frac{(2-\gamma)\delta}{2a}\rp^{\frac{1}{\gamma}} - (2-\gamma) \delta x^{1-\gamma}\\
 &= (2-\gamma)\delta(\lp\frac{(2-\gamma)\delta}{2a}\rp^{\frac{1-\gamma}{\gamma}} - x^{1-\gamma}) \\
 &= (2-\gamma)\delta(x_0^{1-\gamma}- x^{1-\gamma}) \\
 &\leq 0.
\end{align*}

Hence, $f(x) \geq f(x_0)$ for all $x \in \mathbb{R}_{\geq0}$.  Thus, denoting \[b := \lp \lp\frac{2-\gamma}{2}\rp^{\frac{2-\gamma}{\gamma}} - \lp\frac{2-\gamma}{2}\rp^{\frac{2}{\gamma}}\rp,\] our choice of $\delta$ yields
\begin{align*}
 f(x) \geq f(x_0) &\geq a^{-\frac{2-\gamma}{\gamma}}\lp a^{\frac{2-\gamma}{2}}\lp \frac{\epsilon}{b} \rp^{\frac{\gamma}{2}}\rp^{\frac{2}{\gamma}} (-b) 
 =
  -\epsilon.
\end{align*}
\end{proof}

\subsection{Computation of explicit times at which the nonlinear-nudging term in the algorithm improves the convergence rate}\label{sec_comp}

Note that one can compute a time $t_a$ at which $\bu$ is in the absorbing ball (see, e.g. \cite{Foias_Manley_Rosa_Temam_2001, Robinson_2001, Temam_2001_Th_Num}) so that Theorem \ref{ubounds} applies and the exponential decay in \cite{Azouani_Olson_Titi_2014} holds.  The decay of the nonlinear-nudging algorithm is controlled by the exponential decay of the linear-nudging algorithm in \cite{Azouani_Olson_Titi_2014} (see the beginning of the proofs of Theorems \ref{first_conv_theorem},\ref{second_conv_theorem}) can be written explicitly as (for reference, see, e.g., \cite{Carlson_Hudson_Larios_2020})
\begin{align*}
\|\bu(t)-\bv(t)\|_H^2 &\leq \|\bu(t_a)-\bv(t_a)\|_H^2 e^{1+r/2} e^{-\frac{r}{2T}(t-t_a)},
\end{align*}
where $\frac{1}{\nu \lambda_1} < T< \infty$ and 
\begin{align*}
r &= \liminf\limits_{t\to\infty} \int_t^{t+T} \beta - \frac{2c^2}{\nu} \|\bu(s)\|^2 ds 
\\&\geq 
T\beta - \frac{2c^2}{\nu}\lp 2(1+\lambda_1T\nu)\nu G^2\rp > 0,
\end{align*}
and $c$ is the constant from the inequality \eqref{BIN}.

By the assumptions of Theorem \ref{first_conv_theorem}, we need that $\|\bu(t_0)-\bv(t_0)\|_H < R_H$, where $R_H = \min\left\{e^{-\frac{1}{\beta\gamma\lambda_1^{\gamma/2}}}, \lp\frac{\beta\lambda_1^{\gamma/2}}{\beta\lambda_1^{\gamma/2}+1}\rp^{1/\gamma}\right\}$, so we need to choose $t_a$ such that 
\[\|\bu(t_a)-\bv(t_a)\|_H^2 e^{1+r/2} e^{-\frac{r}{2T}(t-t_a)} < R_H^2.\]
Bounding $\|\bu(t_a)\|_H^2$ using Theorem \ref{ubounds} and using the bounds on $r$, we instead find a time $t_a$ such that
\begin{align*}
    \|\bu(t_a)&-\bv(t_a)\|_H^2 e^{1+r/2} e^{-\frac{r}{2T}(t-t_a)} 
    \\ &\leq 
    2(2(\nu)^2G^2+\|\bv(t_a)\|_H^2) e^{1+\beta T} e^{-\beta/2-\frac{c^2}{\nu T}(2(1+\lambda_1T\beta)\nu G^2)(t-t_a)}
    \\ &<
      R_H^2.
\end{align*}

Thus, we determine the nonlinear-nudging system can be initialized from any time $t_0$ such that

\begin{align*}
    t_0 > t_a - \frac{\log \left\{   \frac{R_H^2e^{-1-\beta T}}{2(2(\nu)^2G^2 + \|\bv(t_a)\|_H^2)}      \right\}}{\beta/2 - \frac{c^2}{T\nu}(2(1+\lambda_1\nu T)\nu G^2)}.
\end{align*}
There is no need to observe $\|\bv(t_a)\|_H$ at the fixed time $t_a$; instead, the bound from \cite{Azouani_Olson_Titi_2014}
\[\|\bv(t_a)\|_H^2 \leq e^{-\nu\lambda_1 t_a}\|\bv_0\|_H^2 + \frac{M}{\beta\nu\lambda_1}(1-e^{-\nu\lambda t_a}) := \text{RHS}_H,\]
where $M$ is a constant such that $\|f+ \beta P_\sigma I_h(\bw))\|_H^2 < M$, can be used to choose a time $t_0$ such that
\begin{align*}
    t_0 > t_a - \frac{\log \left\{   \frac{R_H^2e^{-1-\beta T}}{2(2(\nu)^2G^2 + \text{RHS}_H)})    \right\}}{\beta/2 - \frac{c^2}{T\nu}(2(1+\lambda_1\nu T)\nu G^2)}.
\end{align*}
 
For the setting of Theorem \ref{second_conv_theorem}, we have the bound
\begin{align*}
\|\bu(t)-\bv(t)\|_V^2 &\leq \|\bu(t_a)-\bv(t_a)\|^2 e^{\Gamma + 1+r/2}e^{-\frac{r}{2T}(t-t_a)},
\end{align*}
where $\lambda_1 \nu \leq T < \infty$, 
\[r =  \liminf\limits_{t\to\infty} \int_t^{t+T} \frac{1}{2} \lp \beta - \frac{J^2}{\beta} \|A\bu\|_H^2\rp ds > \frac{5}{6} JG > 0,\]
and 
\[ \Gamma = \liminf\limits_{t\to\infty} \int_t^{t+T} \max\left\{\frac{1}{2} \lp \beta - \frac{J^2}{\beta} \|A\bu\|_H^2\rp,0\right\} ds \geq r > 0,\]
with $J = 2c\log(2c^{3/2})+4c\log(1+G)$ and $c$ is the constant dependent on the domain determined from the Brezis-Gallouet inequality.
By the assumptions of Theorem \ref{second_conv_theorem}, we need that $\|\bu(t_0)-\bv(t_0)\|_V < R_V$, where $R_V = \min\left\{e^{-\frac{1}{\beta\gamma\lambda_1^{\gamma/2}}},\lp\frac{\beta\lambda_1^{\gamma/2}}{\beta\lambda_1^{\gamma/2}+1}\rp^{1/\gamma}\right\}$, so we need to choose $t_0$ such that 
\[ \|\bu(t_a)-\bv(t_a)\|_V^2 e^{\Gamma + 1+r/2}e^{-\frac{r}{2T}(t-t_a)} < R_V^2. \]
Again, bounding $\|\bu(t_a)\|_V^2$ using Theorem \ref{ubounds} and using the bounds on $\Gamma$ and $r$, we instead find a time $t_a$ such that
\begin{align*}
    \|\bu(t_a)&-\bv(t_a)\|_V^2 e^{\Gamma + 1+r/2}e^{-\frac{r}{2T}(t-t_a)}
    \\& \leq
      2(2\lambda_1(\nu)^2G^2 + \|\bv(t_a)\|^2) e^{1+\beta T} e^{-\frac{5}{6} GJ(t-t_a)}
    <
      R_V^2
\end{align*}

Then, the nonlinear-nudging system can be initialized from any time $t_0$ such that
\begin{align*}
    t_0 > t_a - \frac{6}{5GJ} \log \left\{ \frac{R_V^2e^{-(1+\beta T)}}{2(2\lambda_1(\nu)^2G^2 + \|\bv(t_a)\|_V^2)}  \right\}.
\end{align*}  
Again, there is no need to observe $\|\bv(t_a)\|_V$ at the fixed time $t_a$, since the bound from \cite{Azouani_Olson_Titi_2014}
\begin{align*}
&\quad
\|\bv(t_a)\|_V^2 
\leq \text{RHS}_V
\\&:=
e^{\frac{54c^4}{(\nu)^3}\lp\frac{1}{\nu}\|\bv_0\|_H^2 + \frac{T}{\nu\beta}M\rp^2\lp\frac{1}{\nu}\|\bv_0\|_V^2 + \frac{M}{\beta\nu\lambda_1} \rp^2}\lp\|\bv_0\|_V^2+\frac{4T}{\nu}M\rp
\end{align*}
where $M$ is the same constant such that $\|f+ \beta P_\sigma I_h(\bw))\|_H^2 < M$, can be used to choose a time $t_0$ such that
\begin{align*}
        t_0 > t_a - \frac{6}{5GJ} \log \left\{ \frac{R_V^2e^{-(1+\beta T)}}{2\left(2\lambda_1(\nu)^2G^2 + \text{RHS}_V\right) }  \right\}.
\end{align*}

\FloatBarrier

\section*{Acknowledgments}
 \noindent
 The authors would like to thank the Isaac Newton Institute for Mathematical Sciences, Cambridge, for support and warm hospitality during the programme ``Mathematical aspects of turbulence: where do we stand?'' where work on this paper was undertaken. This work was supported by EPSRC grant no EP/R014604/1. 
 The research of E.C. was supported in part by NSF GRFP grant no. 1610400 and in part by the Pacific Institute for the Mathematical Sciences (PIMS). The research and findings may not reflect those of the Institute.  E.C. would like to give thanks for the kind hospitality of the COSIM group at Los Alamos National Laboratory where some of this work was completed, and acknowledges and respects the Lekwungen peoples on whose traditional territory the University of Victoria stands, and the Songhees, Esquimalt and WSÁNEĆ peoples whose historical relationships with the land continue to this day.
 The research of A.L. was supported in part by NSF Grants CMMI-1953346 and DMS-2206762.
 The research of E.S.T. was made possible by NPRP grant \#S-0207-200290 from the Qatar National Research Fund (a member of Qatar Foundation), and is based upon work supported by King Abdullah University of Science and Technology (KAUST) Office of Sponsored Research (OSR) under Award No. OSR-2020-CRG9-4336.

\bibliographystyle{abbrv}

\end{document}